\documentclass[reqno,11pt]{amsart}
\usepackage[italian, english]{babel}
\usepackage{amsfonts,amssymb,mathtools}
\usepackage{graphicx}
\usepackage{amsmath}
\usepackage[latin1]{inputenc} 
\usepackage{bbm}
\usepackage{pstricks}
\usepackage{psfrag}
\usepackage{appendix}
\usepackage{amsthm}
\usepackage{dsfont}
\usepackage{stackrel}
\usepackage{eufrak}
\usepackage{mathrsfs}
\usepackage{cite}
\usepackage[a4paper,top=30mm,bottom=30mm,left=30mm,right=30mm]{geometry}
\usepackage{appendix}
\usepackage{enumitem}

\numberwithin{equation}{section}

\newcommand{\dd}{\mathrm{d}}
\newcommand{\ee}{\mathrm{e}}
\newcommand{\prob}{\mathsf{P}}
\newcommand{\probd}{\mathbb{P}}

\newcommand{\Ex}{\mathsf{E}}

\newcommand{\Exd}{\mathbb{E}}
\newcommand{\Rl}{\mathbb{R}}
\newcommand{\Cm}{\mathbb{C}}

\newcommand{\pae}{\mathbb{P}\mbox{-a.e.\ }\omega}
\newcommand{\N}{\mathbb{N}}
\newcommand{\cov}{\mathsf{cov}}

\newcommand{\var}{\mathsf{var}}

\newcommand{\ga}{\alpha}
\newcommand{\gb}{\beta}

\newcommand{\go}{\omega}

\newcommand{\gL}{\Lambda}

\DeclareMathOperator*{\limsmallspace}{\lim~}

\DeclareMathOperator*{\infp}{\inf\phantom{p}}

\newtheoremstyle{mystyle}
  {}
  {}
  {\itshape}
  {}
  {\bfseries}
  {.}
  { }
  {\thmname{#1}\thmnumber{ #2}\thmnote{ (#3)}}

\theoremstyle{mystyle}
\newtheorem{theorem}{Theorem}[section]
\newtheorem{lemma}[theorem]{Lemma}
\newtheorem{proposition}[theorem]{Proposition}
\newtheorem{corollary}[theorem]{Corollary}
\newtheorem{remark}[theorem]{Remark}

\newtheorem{assumption}{Assumption}[section]

\title[Disordered pinning models with contact number constraint]{
Pinning models with contact number constraint: \\ the effect of disorder
}

\author{Giambattista Giacomin and Marco Zamparo}
  
\date{}

\begin{document}

\begin{abstract}
Disordered pinning models are statistical mechanics models built on
discrete renewal processes: renewal epochs in this context are called
\emph{contacts}.  It is well known that pinning models can undergo a
localization/delocalization phase transition: in the localized phase
the typical density of contacts is positive and the largest gap
between contacts is at most of the order of the logarithm of the size
$n$ of the system, whereas the system is void of contacts in the
delocalized phase.  When disorder is absent and the phase transition
is discontinuous, conditioning the contact density to be positive but
smaller than the minimum typical density in the localized phase has
the effect of forcing to create one, and only one, macroscopic gap
between two contacts, while the rest of the configuration keeps the
characteristics of a localized state.  However it is known that, in
the presence of (bounded) disorder, this \emph{big jump phenomenon} is
no longer observed, in the sense that the largest gap in the
conditioned system is $o(n)$. Under minimal integrability conditions
on the disorder we show that the conditioned system is localized in a
very strong sense, in particular the largest gap is $O(\log n)$. The
proof is achieved by exploiting the improved understanding of the
localized phase of disordered pinning models developed in
\cite{cf:companionpaper} and by establishing some refined estimates,
in particular a quenched Local Central Limit Theorem. We also present
an analysis of the effect of disorder on the large deviation rate
function of the contact density and, as an important ingredient for
the generality of our results, we establish a smoothing inequality for
pinning models under minimal integrability conditions, thus
generalizing
 \cite{giacomin2006_2,caravenna2013}.
 \smallskip
 
\noindent  \emph{AMS  subject classification (2020 MSC)}:
60K37,  
82B44, 
60K35, 
60F05  	

\end{abstract}

\maketitle                              


\section{Introduction and main results}

\subsection{The pinning model}
\label{sec:intro-model}

Let $T_1,T_2,\cdots$ be i.i.d.\ $\N$-valued random variables on a
probability space $(\mathcal{S},\mathfrak{S},\prob)$. Here and below
$\N:=\{1,2,\ldots\}$ and $\N_0:=\N\cup\{0\}$. Looking upon
$T_1,T_2,\ldots$ as inter-arrival times, the associated random walk
defined by $S_0:=0$ and $S_i:=T_1+\cdots+T_i$ for $i\in\N$ identifies
the renewal times of a renewal process. Hence the number of renewals
by time $n\in\N_0$ is
\begin{equation*}
  L_n:=\sup\big\{i\in \N_0:\, S_i\le n\big\}\,.
\end{equation*}
The model we consider is built on the renewal sequence
$S:=\{S_i\}_{i\in\N_0}$ and depends on a parameter $h\in\Rl$ and on a
sequence of real numbers $\omega:=\{\omega_a\}_{a\in\N_0}$, which we
call \emph{charges}.  In fact, the \textit{pinning model} is defined
for $n\in\N_0$ by the probability $\prob_{n,h,\omega}$ that satisfies
\begin{equation}
  \frac{\dd\prob_{n,h,\omega}}{\dd\prob}:=\frac{1}{Z_{n,h}(\omega)}\ee^{\sum_{i=1}^{L_n}(h+\omega_{S_i})}\mathds{1}_{\{n\in S\}}\,,
\label{eq:prob_model}
\end{equation}
provided that the \textit{partition function}
$Z_{n,h}(\omega):=\Ex[\ee^{\sum_{i=1}^{L_n}(h+\omega_{S_i})}\mathds{1}_{\{
    n\in S\}}]$ is not 0. The set $\{n \in S\}\in \mathfrak{S}$ is the
event that a renewal time coincides with $n$. Moreover, here and in
what follows, the convention is made that empty sums are 0 and empty
products are 1.

\smallskip

The assumptions on the law of inter-arrival times and the charges are
the following.

\medskip

\begin{assumption}
  \label{assump:p}
 $p(t):=\prob[T_1=t]={\ell(t)}/{t^{\alpha+1}}$ for $t\in\N$ with
  $\alpha\ge 1$ and a slowly varying function at infinity
  $\ell$. Moreover, $p(t)>0$ for all $t\in\N$.
\end{assumption}

\medskip

We recall that a real measurable function $\ell$, defined on
$[1,+\infty)$ in our case, is slowly varying at infinity if
  $\ell(z)>0$ for all sufficiently large $z$ and $\lim_{z \uparrow
    +\infty}\ell(\lambda z)/\ell(z)=1$ for every $\lambda>0$
  \cite{bingham1989}. We assume that $p(t)>0$, i.e., $\ell(t)>0$, for
  all $t\in\N$, which in particular guarantees that
  $Z_{n,h}(\omega)>0$ for every $n$ and every $\omega$.  

  \medskip
  
\begin{assumption}
  \label{assump:omega}
  The charges $\omega:=\{\omega_a\}_{a\in\N_0}$ are sampled
  from a probability space $(\Omega,\mathcal{F},\probd)$ in such a way
  that the canonical projections $\omega\mapsto\omega_a$ form a
  sequence of i.i.d.\ random variables. Furthermore,
  $\int_\Omega\ee^{\eta|\omega_0|}\probd[\dd\omega]<+\infty$ for some
  number $\eta>0$ and $\int_\Omega\omega_0\probd[\dd\omega]=0$. 
\end{assumption}

\medskip

We underline the \emph{quenched} nature of the model, i.e., the fact
that the model is defined for every $n$ once the charges are sampled
and we look for results that hold for typical realizations of them. Of
course the charges introduce a non translation invariant character in
the model: they are meant to take into account \emph{impurities}, more
modernly called \emph{disorder}, that should not be forgotten when
going toward more realistic models.  In our case $h$ represents the
average pinning potential, but the true pinning potential $h+\omega_a$
is not constant and it is site dependent.  With this language, we say
that the model is \emph{pure} or \emph{non disordered} if
$\probd[\{\omega \equiv 0\}]=1$.  The pure pinning model is summarized
by $\prob_{n,h,0}$ and is explicitly solvable (see
\cite{giacomin2007,cf:G-SF,cf:dH}).

\smallskip

There exists an extensive literature on pinning models, with or
without disorder, starting from the $60$s.  We refer to the reviews
\cite{giacomin2007,cf:G-SF,cf:dH,cf:Velenik} for extended
bibliographies and for discussions on the origin and relevance of
pinning models in applied sciences, notably physics and biology.  In
short, the renewal set is interpreted as the ensemble of sites, the
\emph{contacts}, at which an interface, a polymer or another
fluctuating linear structure touches a \emph{defect line}, receiving a
penalty, respectively a reward, at the contact site $a$ if
$h+\omega_a<0$, respectively if $h+\omega_a>0$. With this language,
the size $L_n$ of the contact set is referred to as \emph{contact
number}.

\smallskip

It is possible, and convenient when deriving correlation estimates, to
rewrite the model in terms of the binary random variables
$X_a:=\mathds{1}_{\{a \in S\}}$ for $a\in\N_0$, which
take value 1 at the contact sites and value 0 at the other sites. Note
that $X_0=1$ as $S_0:=0$ by definition. With this notation, we have
that \eqref{eq:prob_model} is equivalent to
\begin{equation}
  \frac{\dd\prob_{n,h,\omega}}{\dd\prob}=\frac{1}{Z_{n,h}(\omega)}\ee^{\sum_{a=1}^n(h+\omega_a)X_a}X_n
  \label{eq:model_binary}
\end{equation}
with $Z_{n,h}(\omega)=\Ex[\ee^{\sum_{i=1}^n(h+\omega_a)X_a}X_n]$, and
the contact number reads $L_n=\sum_{a=1}^nX_a$.

\medskip

\subsection{The free energy density}
\label{sec:fe-density}

Under Assumptions \ref{assump:p} and \ref{assump:omega}, a
superadditivity argument shows that the limit defining the \emph{free
energy density} (we will simply say \emph{free energy} henceforth)
\begin{equation*}
  f(h):= \lim_{n\uparrow\infty}\Exd\bigg[\frac{1}{n}\log Z_{n,h}\bigg]
\end{equation*}
exists and is finite for all $h\in\Rl$ (see \cite{giacomin2007},
Theorem 4.6). We list a number of elementary, but important facts:

\begin{enumerate}[leftmargin=0.7 cm]

\item the function $f$ that maps $h$ to $f(h)$ is non-decreasing and
  convex. Moreover, it is Lipschitz continuous with Lipschitz constant
  equal to 1;
  
\item $Z_{n,h}(\omega)\ge\ee^{h+\omega_n}p(n)$ and
  $Z_{n,h}(\omega)\ge\ee^{\sum_{a=1}^n(h+\omega_a)}p(1)^n$ for all
  $n\in\N$, $h\in\Rl$, and $\omega:=\{\omega_a\}_{a\in\N_0}\in\Omega$,
  so $f(h)\ge \max\{0,h+\log p(1)\}$;

\item $Z_{n,h}(\omega)\le \ee^{\sum_{a=1}^n\max\{0,h+\omega_a\}}$ for
  all $n\in\N$, $h\in\Rl$, and
  $\omega:=\{\omega_a\}_{a\in\N_0}\in\Omega$, so $f(h)\le\int_\Omega
  \max\{0,h+\omega_0\}\probd[\dd\omega]$.
\end{enumerate}

These facts entail that if we set $h_c:=\inf\{h\in\Rl:f(h)>0\}$, then
$-\infty\le h_c<+\infty$, $f(h)=0$ for $h\le h_c$ (when
$h_c>-\infty$), and $\lim_{h\downarrow h_c}f(h)=0$ both when
$h_c>-\infty$ and when $h_c=-\infty$.  We refer to
\cite[Proposition~1.1]{cf:companionpaper} for an almost optimal
criterion to decide whether $h_c>-\infty$ or not: in particular,
$h_c>-\infty$ if $\eta>1/(1+\ga)$ with $\eta$ given in
Assumption~\ref{assump:omega}.  With the definition of $h_c$ at hand,
we say that the model is
\begin{itemize}
\item \emph{delocalized}, or that it is in the \emph{delocalized phase}, if $h<h_c$ (when $h_c>-\infty$);
\item  \emph{critical} if $h=h_c$ (when $h_c>-\infty$);  
\item  \emph{localized}, or that it is in the \emph{localized phase}, if $h>h_c$.
\end{itemize}
Note that $h_c=-\infty$ implies that the model is localized for every $h$.

\medskip

The free energy function in the localized phase has been investigated
in \cite{giacomin2006_1,cf:companionpaper}. We reproduce here the
following findings that are particularly relevant for us (see
\cite[Theorem 1.2]{cf:companionpaper}).

\medskip

\begin{theorem}
  \label{th:Cinfty}
  The free energy $f$ is strictly convex and infinitely
  differentiable on $(h_c,+\infty)$, and the following property holds
  for $\pae$: for every compact set $H\subset (h_c,+\infty)$ and
  $r\in\N_0$
  \begin{equation*}
    \adjustlimits\lim_{n\uparrow\infty}\sup_{h\in H}\bigg|\frac{1}{n}\partial^r_h\log Z_{n,h}(\omega)-\partial_h^r f(h)\bigg|=
    \adjustlimits\lim_{n\uparrow\infty}\sup_{h\in H}\bigg|\Exd\bigg[\frac{1}{n}\partial^r_h\log Z_{n,h}\bigg]-\partial_h^r f(h)\bigg|=0\,.
  \end{equation*}
  Moreover, $f$ is of class Gevrey-3 on $(h_c,+\infty)$, i.e., for
  every compact set $H\subset (h_c,+\infty)$ there exists a constant
  $c>0$ such that for all $r\in\N$
 \begin{equation*}
 \sup_{h \in H} \Big| \partial_h^r f(h)\Big| \le c^r (r!)^3\,.
 \end{equation*}

\end{theorem}

\medskip

We will refer to the first derivative $\rho:=\partial_hf$ of $f$ in
the localized phase as the \textit{contact density} because, bearing
in mind that $\Ex_{n,h,\omega}[L_n]=\partial_h\log Z_{n,h}(\omega)$,
Theorem~\ref{th:Cinfty} shows that the following holds for $\pae$: for
every $h>h_c$
 \begin{equation}
   \lim_{n\uparrow\infty}\Ex_{n,h,\omega}\bigg[\frac{L_n}{n}\bigg]=\partial_hf(h)=:\rho(h)\,.
   \label{eq:contact_fraction}
 \end{equation}
The second derivative $v:=\partial_h^2f$ of $f$ in the localized phase
is the limiting scaled variance of the contact number. In fact, as
$\Ex_{n,h,\cdot}[(L_n-\Ex_{n,h,\cdot}[L_n])^2]=\partial_h^2\log
Z_{n,h}$, Theorem~\ref{th:Cinfty} implies that also the following is
valid for $\pae$: for every $h>h_c$
\begin{equation}
  \lim_{n\uparrow\infty}\Ex_{n,h,\omega}\bigg[\bigg(\frac{L_n-\Ex_{n,h,\omega}[L_n]}{\sqrt{n}}\bigg)^{\!\!2}\bigg]=\partial_h^2f(h)=:v(h)\,.
\label{eq:lim_vh}
\end{equation}
We see that $\rho$ takes values in $[0,1]$ and the strict convexity of
the free energy in the localized phase implies that $\rho$ is strictly
increasing and $v$ strictly positive over $(h_c,+\infty)$. Then we
infer that $\rho(h)>0$ for all $h>h_c$, in agreement with the
intuitive notion of localized phase and in contrast with the fact that
$\lim_{n\uparrow\infty}\Ex_{n,h,\omega}[L_n/n]=0$ in the delocalized
phase. We note that $\lim_{h\uparrow+\infty}\rho(h)=1$ since combining
convexity and differentiability of $f$ with the bound
$f(h)\ge\max\{0,h+\log p(1)\}$ we deduce that $f(0)\ge\max\{0,h+\log
p(1)\}-h\rho(h)$ for $h>h_c$, which in turn gives
$\liminf_{h\uparrow+\infty}\rho(h)\ge 1$.

The limit $\rho_c:=\lim_{h\downarrow h_c}\rho(h)$ is of main
importance for us. When $h_c=-\infty$, the above bound
$f(0)\ge\max\{0,h+\log p(1)\}-h\rho(h)$ for $h>h_c$ shows that
$\limsup_{h\downarrow h_c}\rho(h)\le 0$, and hence $\rho_c=0$. When
$h_c>-\infty$, the model undergoes a localization/delocalization phase
transition that is \textit{continuous} if $\rho_c=0$ and
\textit{discontinuous}, or \textit{first order}, if $\rho_c>0$.

\medskip

\subsection{Large deviations for the contact number}
\label{sec:LDP}

The function $I_h$ that maps $r\in\Rl$ to
\begin{equation}
\label{eq:I_h}
I_h(r):=\sup_{k\in\Rl}\big\{r(k-h)-f(k)+f(h)\big\}
\end{equation}
turns out to be a rate function in a Large Deviation Principle
(LDP). We note that $I_h$ is a way to write the convex conjugate of
$f(h+\cdot)-f(h)$, the latter being the limiting scaled cumulant
generating function of the contact number.  In fact, we have the
following result under Assumptions~\ref{assump:p}
and~\ref{assump:omega} (see \cite[Theorem 1.4 and Corollary
  1.11]{hollander2024}).

\medskip

\begin{proposition}
\label{th:LD}
  The following bounds defining a full LDP for the contact density
  hold for $\pae$:
\begin{enumerate}[leftmargin=0.7 cm, itemsep=1ex,label=({\roman*})]
\item $\liminf_{n\uparrow\infty} (1/n)\log\prob_{n,h,\omega}[L_n/n\in G] \ge-\inf_{r\in G} I_h(r)$ for all $h\in\Rl$ and $G\subseteq\Rl$ open;
\item $\limsup_{n\uparrow\infty} (1/n)\log\prob_{n,h,\omega}[L_n/n\in F] \le-\inf_{r\in F} I_h(r)$ for all $h\in\Rl$ and $F\subseteq\Rl$ closed.  
\end{enumerate}
\end{proposition}

\medskip

One can show that $I_h(r)=+\infty$ for $r\notin[0,1]$ in a number
of ways, so the effective domain of $I_h$, i.e., the region where
$I_h$ is finite, is contained in the closed interval $[0,1]$.  This
is, for example, a consequence of part $(i)$ in
Proposition~\ref{th:LD} since $0\le L_n\le n$.  Of course a bounded
effective domain for $I_h$ directly entails that, in the language of
large deviation theory, the rate $I_h$ is \emph{good}, i.e., its
sub-level sets are compact.

The effective domain of $I_h$ is exactly the closed interval
$[0,1]$. In fact $I_h$ is convex and finite at the boundary of such
interval: $I_h(0)=f(h)$ and $I_h(1)=f(h)-h-\log p(1)$. To see this we
note that both $f(k)$ and $k-f(k)$ are non-decreasing with respect to
$k$ with $\lim_{k\downarrow-\infty}f(k)=0$. Moreover, using that
$p(t)\le p(1)^{2-t}$ for every $t\in\N$ one can demonstrate that
$Z_{n,k}(\omega)\le
p(1)\ee^{k+\omega_n}\prod_{a=1}^{n-1}[1/p(1)+\ee^{k+\omega_a}p(1)]$
for all $n\in\N$, $k\in\Rl$, and
$\omega:=\{\omega_a\}_{a\in\N_0}\in\Omega$, so $f(k)\le\int_\Omega
\log[1/p(1)+\ee^{k+\omega_0}p(1)]\probd[\dd\omega]$. Combining the
latter with the bound $f(k)\ge k+\log p(1)$ one finds that
$\lim_{k\uparrow+\infty}[k-f(k)]=-\log p(1)$.

Of major interest is the shape of $I_h$ over the interior of its
effective domain. To determine it, we recall that $\rho:=\partial_hf$
is strictly increasing over $(h_c,+\infty)$ from $\rho_c$ to 1, so it
has an inverse function $\imath_\rho$ that maps the interval
$(\rho_c,1)$ onto the interval $(h_c,+\infty)$. Since the free energy
$f$ is infinitely differentiable in the localized phase by Theorem
\ref{th:Cinfty}, $\imath_\rho$ is infinitely differentiable because of
the inverse function theorem.  One can easily verify that
$\imath_\rho$ allows us to write $I_h$ for every $h\in\Rl$ as follows:
\begin{equation}
  I_h(r)=\begin{cases}
r(h_c-h)+f(h) & \mbox{if $r\in[0,\rho_c]$}\\
r[\imath_\rho(r)-h]-f\circ\imath_\rho(r)+f(h)  & \mbox{if $r\in(\rho_c,1)$}\\
f(h)-h-\log p(1)  & \mbox{if $r=1$}
\end{cases}
\label{eq:I_h_esplicita}
\end{equation}
with the convention that $r(h_c-h)=0$ when $r=0$ and
$h_c=-\infty$.

The rate function $I_h$ of the contact density turns out to be
continuously differentiable on $(0,1)$ and infinitely differentiable
and strictly convex on $(\rho_c,1)$, the strict convexity being
manifest from the formula $\partial_r^2
I_h(r)=1/\partial_h^2f(\imath_\rho(r))>0$ valid for all
$r\in(\rho_c,1)$.  Furthermore, $I_h$ is of class Gevrey-3 on
$(\rho_c,1)$, inheriting this property from $f$ thanks to Theorem
\ref{th:Cinfty} since Gevrey classes are closed under inversion and
composition (closure under composition is already in the work by
Maurice Gevrey \cite{cf:Gevrey} and one can find closure under
differentiation and passing to reciprocal, hence also about inversion,
in \cite[pages~94-95]{cf:primer}, along with the references to the
original literature). Importantly, when $\rho_c>0$, i.e., when the
model undergoes a first order phase transition, the rate function
$I_h$ exhibits an affine stretch terminating at $\rho_c$, so two
different regimes of large deviations appear corresponding to strict
and non-strict convexity of $I_h$. A natural question concerns the
meaning of these two regimes.

\medskip

\begin{remark}
\rm{We stress that Proposition \ref{th:LD}, like Theorem
  \ref{th:Cinfty} and all the results that follow, is uniform with
  respect to the free parameter $h$: we identify a set of charges of
  full probability for which the results hold for every $h\in\Rl$ or
  every $h>h_c$ depending on the situation. In the case of Proposition
  \ref{th:LD}, such a uniformity is not immediate from the results in
  \cite{hollander2024}, where there is no a free parameter like
  $h$. What \cite[Theorem 1.4 and Corollary 1.11]{hollander2024}
  really state is that for each $h\in\Rl$ there exists a set of
  charges $\Omega_h\in\mathcal{F}$ with $\probd[\Omega_h]=1$ such that
  the large deviation bounds $(i)$ and $(ii)$ hold for all
  $\omega\in\Omega_h$.  But then one realizes that the set $\Omega_0$
  corresponding to no free parameter works for all $h$. In fact, the
  model $\prob_{n,h,\omega}$ can be regarded as an exponential change
  of measure of the model $\prob_{n,0,\omega}$:
\begin{equation*}
  \frac{\dd \prob_{n,h,\omega}}{\dd \prob_{n,0,\omega}}=\frac{\ee^{hL_n}}{\Ex_{n,0,\omega}[\ee^{hL_n}]}\,.
\end{equation*}
Since the change of measure is governed by the contact number, the
full LDP for the contact density with good rate function in the model
$\prob_{n,0,\omega}$ with a given $\omega$ implies a full LDP with
good rate function in the model $\prob_{n,h,\omega}$ with the same
$\omega$ (see \cite[Theorem III.17]{hollander_LDP}). It therefore
follows that $\Omega_h\supseteq\Omega_0$ because of the uniqueness of
the rate function.}
\end{remark}

\medskip

\subsection{The big jump phenomenon for pure models}

The pure model is explicitly solvable, as mentioned earlier.  In the
absence of disorder we have $h_c=0$ and the free energy $f(h)$ for
$h>h_c$ is the unique positive real number that solves the equation
$\Ex[\ee^{-f(h) T_1}]=\ee^{-h}$ (see \cite[Proposition
  1.1]{giacomin2007}). The real analytic implicit function theorem
shows that the function $f$ is analytic for $h>h_c$, and one can
easily verify that $\rho_c:=\lim_{h\downarrow
  h_c}\rho(h)=\lim_{h\downarrow h_c}\partial_hf(h)=1/\Ex[T_1]$ if
$\Ex[T_1]<+\infty$ and $\rho_c=0$ if $\Ex[T_1]=+\infty$.  So there is
a localization/delocalization phase transition, which is first order
when $\Ex[T_1]<+\infty$. In this case the rate function $I_h$ of the
contact density has an affine stretch terminating at $\rho_c$
according to formula \eqref{eq:I_h_esplicita} and two different
regimes of large deviations occur.  In all cases the rate function is
strictly convex and analytic in the region $(\rho_c,1)$.

\smallskip

An explanation of the two different regimes of large deviations is
achieved by investigating the maximal gap between contacts
$M_n:=\max\{T_1,\ldots,T_{L_n}\}$ conditional on a given number $l$ of
contacts. The following proposition, which is proved in Appendix
\ref{proof:no disorder}, lists the main results of such investigation.
We note that the presence of the parameter $h$ in this proposition,
and in Theorem \ref{th:main} below, is just a presentation choice
because the sharp conditioning disables such parameter:
$\prob_{n,h,\omega}[\,\cdot\,|L_n=l]$ is independent of $h$ whatever
$n$, $l$, and $\omega$ are. Let $\imath_\rho$ be the inverse function
of the contact density $\rho$ as previously defined.

\medskip

\begin{proposition}
  \label{prop:no_disorder}
  The following conclusions hold:
    \begin{enumerate}[leftmargin=0.7 cm, itemsep=1ex,label=({\roman*})]
  \item for every $h\in\Rl$, closed set $R\subset(\rho_c,1)$, and $\epsilon>0$
\begin{equation*}
  \adjustlimits\limsmallspace_{n\uparrow\infty} \sup_{l\in nR\cap \N }
  \,\prob_{n,h,0}\bigg[\bigg|\frac{M_n}{\log n}-\frac{1}{f\circ\imath_\rho(l/n)}\bigg|>\epsilon\bigg|L_n=l\bigg]=0\,;
\end{equation*}
\item if $\rho_c>0$, then for every $h\in\Rl$, closed set $R\subset(0,\rho_c)$, and $\epsilon>0$
  \begin{equation*}
  \adjustlimits\limsmallspace_{n\uparrow\infty} 
  \sup_{l\in nR\cap \N }\,\prob_{n,h,0}\bigg[\bigg|\frac{M_n}{n}-\left(1-\frac{l/n}{\rho_c}\right)\bigg|>\epsilon\bigg|L_n=l\bigg]=0\,.
\end{equation*}
\end{enumerate}
\end{proposition}

\medskip

Proposition \ref{prop:no_disorder} states that the regime of large
deviations corresponding to the affine stretch of $I_h$ involves a
\textit{big jump phenomenon}, i.e., $M_n$ is $O(n)$, as the most
likely mechanism to achieve a small number of contacts. On the
contrary, $M_n$ is \textit{almost microscopic}, i.e., $O(\log n)$, in
the regime of large deviations where $I_h$ is strictly convex. The big
jump phenomenon is a well-known feature of heavy-tailed random
variables, such as the inter-arrival times used here: large values of
sums of independent heavy-tailed random variables are attained by a
large value of just one of the terms in the sum
\cite{armendariz2011,cf:DDS,foss2011,cf:MN}.

 \medskip

\subsection{Main results: smoothing, contact number conditioning, and disorder effect}

One interesting question concerns the effect of disorder on the
conditioned model we consider and, in particular, on the big jump
phenomenon. The presence or the lack of an affine stretch in the graph
of $I_h$ is related to the event $\rho_c>0$, which in turn depends on
the occurrence of a phase transition, i.e., on $h_c>-\infty$, and on
the differentiability properties of the free energy $f$ at $h_c$: in
fact, by convexity $\rho_c:=\lim_{h\downarrow h_c}\rho(h)$ with
$\rho:=\partial_hf$ is the right derivative of $f$ at $h_c$, whereas
the left derivative is 0.  For this reason, we propose first a
smoothing inequality under Assumptions \ref{assump:p} and
\ref{assump:omega}, generalizing previous works (see
\cite{giacomin2006_2,giacomin2006_3,caravenna2013}).

\medskip

\begin{theorem}
  \label{th:smoothing}
  Suppose that $\int_\Omega\omega_0^2\,\probd[\dd\omega]>0$.  If
  $h_c>-\infty$, then there exists a constant $c>0$ such that for all
  sufficiently small $\delta>0$
  \begin{equation*}
   f(h_c+\delta)\le c\,\delta^2\,. 
  \end{equation*}
\end{theorem}

\medskip

We recall that the smoothing bound in \cite[Theorems~1.5 and
  1.8]{caravenna2013} works for charges of the form $\beta\omega_a$
with $\int_\Omega\omega_0^2\,\probd[\dd\omega]=1$ and $\gb \in (0,
\xi/2)$, $\beta$ being a free parameter that controls the amplitude of
the charges and $\xi>0$ being a number such that $\int_\Omega
\ee^{z\omega_0}\probd[\dd\omega]<+\infty$ for all $z\in(-\xi,\xi)$. In
our setting, where there is no $\beta$ but the variance of the charges
can be any, the hypothesis $\beta<\xi/2$ becomes $\eta>2$ with $\eta$
as in Assumption \ref{assump:omega}. Therefore
Theorem~\ref{th:smoothing} is already proved in \cite{caravenna2013}
if $\eta>2$, but here we extend the smoothing bound to the case
$\eta\le 2$ and, in particular, to $1/(1+\ga)<\eta \le 2$ for which we
have $h_c>-\infty$ as recalled in Section~\ref{sec:fe-density}.

\smallskip

We are now in the position to state that, in the presence of disorder,
$\rho_c=0$ irrespective of whether $\Ex[T_1]$ is infinite or not. In
fact, either disorder washes out the phase transition entailing
$h_c=-\infty$ or the phase transition is still present but it is
necessarily continuous as the right derivative of $f$ at $h_c$ is 0
according to Theorem \ref{th:smoothing}.  Consequently, the affine
stretch in the graph of $I_h$ is washed out by disorder and the two
different regimes of large deviations of the pure model merge into a
single one.  The following corollary of formula
\eqref{eq:I_h_esplicita} and Theorem \ref{th:smoothing} collects what
we know about the rate function $I_h$.


\medskip

\begin{corollary}
  \label{th:strict_convexity}
   Suppose that $\int_\Omega\omega_0^2\,\probd[\dd\omega]>0$. For
   every $h\in\Rl$ the rate function $I_h$ is strictly convex,
   infinitely differentiable, and of class Gevrey-3 on $(0,1)$.
\end{corollary}

\medskip

Our main finding is that disorder prevents any form of big jump
phenomena, as stated by the following theorem.

 \medskip

\begin{theorem}
  \label{th:main}
  Suppose that $\int_\Omega\omega_0^2\,\probd[\dd\omega]>0$. The
  following property holds for $\pae$: for every $h\in\Rl$ and closed
  set $R\subset(0,1)$ there exists a constant $c>0$ (independent of
  $\omega$) such that
  \begin{equation*}
    \adjustlimits\limsmallspace_{n\uparrow\infty} \sup_{l\in nR\cap \N}\, \prob_{n,h,\omega}\big[M_n>c\log n\big| L_n = l\big]=0\,.
  \end{equation*}
\end{theorem}

\medskip

Theorem~\ref{th:main} substantially strengthens the results in
\cite{giacomin2020} where in particular it is shown (assuming bounded
disorder) that for every closed set $R\subset(0,1)$ and $\epsilon>0$
and for $\pae$
\begin{equation}
\label{eq:fromGH}
    \adjustlimits\limsmallspace_{n\uparrow\infty} \sup_{l\in nR\cap \N }\,\prob_{n,h,\omega}\big[M_n> \epsilon n\big|L_n=l\big]=0 \,.
  \end{equation}
In words, \eqref{eq:fromGH} just states that there is no macroscopic
big jump in the presence of disorder, i.e., the largest gap between
contacts is not macroscopic. Theorem~\ref{th:main} instead says that
the largest gap between contacts is almost microscopic and it is of
the same order of the largest gap between contacts in the localized
phase of the (pure or disordered) pinning model (see
\cite[Theorem~2.5]{giacomin2006_1} and
\cite[Proposition~1.9]{cf:companionpaper}). Other results analogous to
\eqref{eq:fromGH} are proved in \cite{giacomin2020}, notably in the
directions of generalized pinning models proposed in physics. Also
these results can be strengthened to obtain that the largest gap
between contacts is almost microscopic and
Section~\ref{sec:generalizations} is devoted to this.

\smallskip

We can refine Theorem~\ref{th:main} in the direction of saying that,
not only there is no trace of the big jump, but also that the contact
density is constant through the system even down to scale lengths that
diverge faster than $\log n$. This is the content of the next theorem.

\medskip

\begin{theorem}
\label{th:smallscales} 
  The following property holds for $\pae$: for every $h\in\Rl$ and
  $\epsilon>0$, closed set $R\subset(0,1)$, and sequence
  $\{\zeta_n\}_{n\in\N}$ of positive numbers such that
  $\lim_{n\uparrow\infty}\zeta_n=+\infty$
  \begin{equation*}
    \adjustlimits\limsmallspace_{n\uparrow\infty}\sup_{l\in nR\cap\N}\,
    \prob_{n,h,\omega}\Bigg[\max_{\substack{0\le i<j\le n\\j-i\ge \zeta_n\log n}}\bigg|\frac{1}{j-i}\sum_{a=i+1}^jX_a-\frac{l}{n}\bigg|>\epsilon\Bigg|L_n=l\Bigg]=0\,.
\end{equation*}
  \end{theorem}

\medskip

Proving results like those in Theorem~\ref{th:main} and in
Theorem~\ref{th:smallscales} on the conditional measure requires a
good control of the probability mass function of $L_n$. In fact, a
crucial ingredient is a quenched local Central Limit Theorem (CLT) for
the contact number in the localized phase. This local result is
interesting in itself and we state it here among the main results.

\medskip

\begin{theorem}
  \label{th:LCLT}
  The following property holds for $\pae$: for every $H\subset (h_c,+\infty)$ compact
  \begin{equation*}
  \adjustlimits\limsmallspace_{n\uparrow\infty}
  \sup_{h\in H}\,\sup_{l\in\N_0}\Bigg|\sqrt{2\pi v(h) n}\,\prob_{n,h,\omega}[L_n=l]-\exp \bigg\{\!-\frac{(l-\Ex_{n,h,\omega}[L_n])^2}{2 v(h) n}\bigg\}\Bigg|=0
  \end{equation*}
  with $v(h):=\partial_h^2f(h)>0$.
\end{theorem}

\medskip

Theorem~\ref{th:LCLT} is established by estimating the characteristic
function of the relevant random variable.  This is a well-known
approach for independent random variables and it has been extended for
spin averages in non-disordered Gibbs models under suitable conditions
on the interaction and assuming that the CLT holds: the case of local
interactions is discussed in \cite{cf:DT77} and the case of
infinite-range pair interactions is tackled in
\cite{campanino1979}. The pinning model can be recast in a binary spin
model language and the contact number becomes the total magnetization
(see, e.g., \cite{zamparo2019}), but the potentials involved contain
arbitrarily many body interactions (this will be taken up again in
Section~\ref{sec:firstorderetc}).  Moreover, our model is
disordered. So we cannot apply the results in
\cite{cf:DT77,campanino1979} and we have to exploit the specific
structure of one-dimensional pinning models.

\medskip

\subsection{First order transitions, big jump phenomena, and phase separation droplet}
\label{sec:firstorderetc}

Big jump phenomena are widespread and they are the crucial mechanisms
in problems studied in depth from fundamental mathematics to applied
sciences: in fact, they have attracted attention in the mathematical
community in connection with several condensation phenomena under
conditioning (see \cite{ferrari2007} and references therein) and
anomalous transport (see \cite{zamparo2023} and references
therein). We refer to \cite{cf:godreche,vezzani2019,zamparo2021} and
references therein for big jump regimes in the physical literature.
Investigating the stability of big jump phenomena in random media is a
special case of the general and challenging issue of understanding the
effect of disorder on phase transitions and critical phenomena (see,
e.g., \cite[Chapter~5]{cf:G-SF} and references therein). To the best
of our knowledge, this has been done only for condensation phenomena
in the zero-range process with site-dependent jump rates
\cite{cf:GCS,cf:molino}, exploiting the fact that this process has a
factorized steady state with single-site marginals that depend solely
on the local disorder.  Coming to the pinning context, the way in
which disorder enters the model is more spatially structured than in
the zero-range process and this is one of the challenges of the
problem we face.  The problem for pinning models is challenging to the
point that the most natural generalization of the model we tackle, and
(to a large extent) solve here, is widely open as we explain in
Subsection~\ref{sec:HD}.



\smallskip

We now discuss further these and  related issues. 

\subsubsection{Analogy with the phase separation droplet phenomenon}
Proposition~\ref{prop:no_disorder} and our main result,
Theorem~\ref{th:main}, are actually results about \emph{canonical
Gibbs measures} and in this sense there is a direct link with the
problem much studied in statistical mechanics
\cite{cf:Wulff1,cf:Wulff2,cf:Wulff3} that we now outline. Consider the
(non disordered) nearest-neighbor ferromagnetic Ising model, without
external field, in the box
$\Lambda_n:=\{-n,\ldots,n\}^d\subset\mathbb{Z}^d$ with dimension $d\ge
2$ and at a temperature below the critical one.  For such value of the
temperature the infinite-volume Gibbs measure is not unique and there
are two (extremal) translational invariant measures, one with average
magnetization $m_+>0$ and the other one with average magnetization
$m_-=-m_+<0$: this phenomenology is related to the first order
character of the phase transition. If $+1$ boundary conditions are
imposed, then as $n$ goes to infinity the $\Lambda_n$ Gibbs measure
converges to the extremal Gibbs measure with average magnetization
$m_+$.  But if we condition the same model to have a given total
magnetization which is asymptotically equivalent to $m \vert \gL_n
\vert$ with $m \in (m_-, m_+)$, then sending $n$ to infinity the
system rearranges in a phase separation pattern so that at a
macroscopic level, i.e., when the box is rescaled to $[-1,1]^d$, the
system is partitioned into a region of the minus phase and a region of
the plus phase (and an interface or hypersurface separating the two
phases). In particular, if $m$ is not too far from $m_+$, then there
is a droplet -- the \emph{Wulff crystal} -- of the $m_-$ phase in a
sea of the $m_+$ phase. This is analogous to the big jump phenomenon
in the one dimensional pinning model and, in both cases, this
phenomenon is directly linked to the presence of a first order phase
transition. Of course in the pinning model the two phases are the
localized phase $h>h_c$, with contact density $\rho(h)$, and the
delocalized phase in which not only the contact density is zero, but
the system is really void of points. Instead in the Ising model the
two phases are perfectly symmetrical: statistically they can be mapped
one into the other by switching the spin signs. This difference is
discussed further in Subsection~\ref{sec:longrangemanybody}.

First order phase transitions are present in Ising models also when
$d=1$, but we have to step to long-range interactions and in this
context the droplet phenomenon has been studied in \cite{cf:CMP2017}
where it is shown that, under suitable hypotheses, it appears a phase
separation droplet that takes precisely the shape of a single interval
-- a big jump -- of the minus phase in the sea of the plus phase.

Results addressing the persistence of the phase separation droplet in
disordered Ising models have been developed in the context of diluted
models \cite{cf:Wouts}. But the case we consider is more closely
related to external random field as disorder. In the nearest-neighbor
Ising context that we have just outlined, it is known that the first
order phase transition persists if $d \ge 3$ and it does not if $d=2$
(see \cite{cf:AW,cf:BK-Ising3d,cf:DingXia,cf:DingZhuang} and
references therein). For the $d=1$ long-range case the situation is
more complex and the first order transition may or may not withstand
disorder \cite{cf:AW,cf:DingHuandMaia}. In this sense the pinning
model has a parallel in the $d=2$ Ising model and in some $d=1$
long-range cases, with the first order transition washed out by
disorder.

We present results on the canonical measure of the pinning model that
hold uniformly in the volume and down to the microscopic scale,
showing that the system is localized in a very strong sense, i.e., in
the sense that there is no trace of the delocalized phase: in making
this statement we play of course on the very different nature of the
two phases in the pinning model.  We are not aware of a detailed
pathwise analysis of the canonical Gibbs Ising measure with random
external field, even if it may be approachable with the recent
considerable progress in this domain
\cite{cf:AHP,cf:DingHuandMaia,cf:DingXia}.

\subsubsection{More about disorder smoothing and the connection with Ising models}
\label{sec:longrangemanybody}
As mentioned in Section \ref{sec:intro-model}, it is possible to write
the pinning model in terms of the binary random variables
$X_0,X_1,\ldots$, with the minor difference with respect to the Ising
case that the single spin state space is $\{0,1\}$ instead of
$\{-1,1\}$. The major difference is that the reference measure is not
the product measure, but a discrete point process. By this we mean
that, under the law $\prob$, the sequence $\{X_a\}_{a\in\N_0}$ is a
regenerative process with heavy-tailed inter-arrival distribution
(which is precisely the probability mass function $p$ of
Assumption~\ref{assump:p}), and of course we interpret a site $a$ such
that $X_a=1$ as a renewal epoch. In order to compare with the Ising
model, the variables $X_0,X_1,\ldots$ should be i.i.d.\ Bernoulli
variables: in this case the inter-arrival times would be geometric
random variables. We can change the reference measure to the measure
of i.i.d.\ Bernoulli variables, but the Hamiltonian of the pinning
model passes from containing only one body interactions to arbitrarily
many body interactions (see \cite[formula (4)]{zamparo2019}). The
structure and properties of this long-range and many-body Hamiltonian
and the associated specification has been only partially studied and
here we limit ourselves to the following considerations:
 

\begin{enumerate}[leftmargin=0.7 cm]
\item of course the long-range and many-body nature of the pinning
  interactions is at the very origin of the
  localization/delocalization phase transition, with the two phases
  that are very asymmetrical (one concentrates on all
  zeros). Despite this asymmetry, we point out the partial analogy
  with the 
   one-dimensional Ising model with
  long-range two-body interactions, which exhibits a first order phase
  transition if the interactions decay slow enough. It is precisely in
  this context that the droplet model has been investigated (see
  \cite{cf:CMP2017}), as well as the effect of a random external field
  on the transition (see in particular
  \cite{cf:AW,cf:COP1,cf:DingHuandMaia} that include also recent
  progress);
\item this discussion brings to the spotlight the analogy between the
  smoothing phenomenon for pinning models, i.e.,
  Theorem~\ref{th:smoothing} and
  \cite{giacomin2006_2,giacomin2006_3,caravenna2013}, and the
  \emph{Imry--Ma smoothing phenomenon} first mathematically
  established for Ising spin systems (and beyond) in the seminal work
  by Aizenman and Wehr \cite{cf:AW} (see also the recent progresses
  \cite{cf:AHP,cf:DingXia}).  The arguments leading to these two
  smoothing phenomena appear to be different, but there might be a
  bridge between them (beyond the fact that both mechanisms wash out
  the first order transition). A path for linking the two smoothing
  phenomena could be via studying larger classes of long-range
  many-body interactions for Ising spin systems in one dimension.
\end{enumerate}


\subsubsection{On higher dimensional pinning models and open problems}
\label{sec:HD}
There is a natural generalization of the pinning model that consists
in considering pinning for renewals in dimensions larger than one (see
\cite{cf:GO04,cf:GK} and references therein).  This is directly
stimulated by the applications and by the very first motivation, at
least from the historical viewpoint, for the pinning model: the DNA
denaturation transition. In fact, in the biophysics community the
pinning model goes under the name of \emph{Poland-Scheraga model}
\cite{cf:PS}.  And, always in the biophysics community, also a
\emph{generalized Poland-Scheraga model} \cite{cf:GO04} has been
introduced to account for mismatches in the pairing of the two DNA
strands, and, in mathematical terms, this model is precisely a pinning
model based on a two dimensional renewal process \cite{cf:GK}. For
this higher dimensional renewal system one can actually show that the
localization transition is present also in the disordered model (see
\cite{cf:BGK2,cf:legrand} and references therein), but the pure model has a
regime in which a big jump appears if the two stands are of
sufficiently different lengths. We stress that the big jump phenomenon
in the generalized Poland-Scheraga model is present without resorting
to conditioning, but it demands that the two polymer strands have
different lengths, and this is analogous to forcing the longer strand
to have less contacts than the \emph{natural amount}. Whether disorder
washes out the big jump phenomenon in this generalized context is
 open.

\smallskip

We conclude this section by pointing out that the main results we
present, Theorems~\ref{th:main} and \ref{th:smallscales}, hold
uniformly in contact densities belonging to closed subsets of $(0,1)$.
We do believe that the results should hold also for closed subsets of
$(0,1]$, i.e., for contact densities that can be arbitrarily close to
the fully packed system.  This is particularly intuitive for
Theorem~\ref{th:main}: in fact if the contact number is larger than $
n- c \log n$ for a constant $c>0$, then there is no room for a contact
gap of more than $c \log n$.  At the expense of a nontrivial amount of
technical work and supposing that disorder is bounded, we can extend
our results to contact densities that approach $1$ as $n$ becomes
large, but we are unable to cover all disorder laws in Assumption
\ref{assump:omega}. This appears to be a limit of the local CLT
approach we adopt.


\subsection{Organization of the rest of the paper}
 Section~\ref{sec:smoothing} reports the proof of Theorem
 \ref{th:smoothing}. Section~\ref{sec:noBJ} contains the proof of the
 local CLT in Theorem \ref{th:LCLT}, after that few relevant results
 from the companion paper \cite{cf:companionpaper} are
 recalled. Theorem \ref{th:LCLT} is one of the main ingredients for
 the proofs of Theorem~\ref{th:main} and of
 Theorem~\ref{th:smallscales}, which can be found respectively in
 Sections~\ref{sec:gaps} and \ref{sec:meso}.  In
 Section~\ref{sec:generalizations} a model proposed in biophysics for
 the denaturation transition of DNA chains with a ring structure
 (\emph{plasmids}) is considered, and the big jump phenomenon for this
 model is discussed in the light of Theorem~\ref{th:main}. Also
 results in which the sharp conditioning $L_n=l$ is replaced by a soft
 one are given. The Appendix is devoted to the proof of
 Proposition~\ref{prop:no_disorder}.


\section{Smoothing of disorder under subexponential conditions}
\label{sec:smoothing}

The proof of Theorem \ref{th:smoothing} follows the line of
\cite{giacomin2006_2,giacomin2006_3,caravenna2013}, and notably
exploits the \emph{rare stretch strategy} applied to a model in which
the charges are \emph{tilted}, according to the distribution
\eqref{eq:omegatilt} below, in order to achieve mean $h$, while in the
original model the mean $h$ is achieved by a shift, i.e., by adding
$h$ to $\go_a$ for every $a$. In fact, we explicitly exploit the
optimized result in \cite{caravenna2013} for the free energy smoothing
of the model with tilted disorder.  But tilting is equivalent to
shifting only in the Gaussian case, and this is used in
\cite{giacomin2006_3} to give a very compact proof.  For general
disorder the control of shifting by tilting is obtained in
\cite{giacomin2006_2} supposing that the charges are bounded.  In
\cite{caravenna2013} the boundedness is weakened to
Assumption~\ref{assump:omega}, at least if the variance of the charges
is below a certain threshold, which is infinite if the tail of the
charge law has super-exponential decay.  We generalize these results
by removing any condition beyond Assumptions \ref{assump:p} and
\ref{assump:omega}. From the technical viewpoint, unlike the arguments
in \cite{giacomin2006_2,caravenna2013}, our argument is direct in the
sense that we do not go through the convex conjugate of the free
energy function.

\begin{proof}[Proof of Theorem \ref{th:smoothing}]
For $s\in(0,\eta)$, with $\eta$ as in Assumption \ref{assump:omega},
and $h\in\Rl$ define
\begin{equation*}
f_s(h):=\limsup_{n\uparrow\infty}\frac{1}{n}\int_\Omega\log Z_{n,h}(\omega)\,\probd_{n,s}[\dd\omega]\,,
\end{equation*}
where we have introduced the \emph{tilted measure}
\begin{equation}
\label{eq:omegatilt}
  \probd_{n,s}[\dd\omega]:=\frac{\ee^{s\sum_{a=1}^n\omega_a}\probd[\dd\omega]}
       {\int_\Omega\ee^{s\sum_{a=1}^n\omega_a'}\probd[\dd\omega']}\,.
\end{equation}
A smoothing inequality with respect to this disorder tilt is provided
under Assumptions \ref{assump:p} and \ref{assump:omega} by
\cite[Theorem 1.5]{caravenna2013}, which states in particular that if
$h_c>-\infty$, then there exists $s_o\in(0,\eta)$ such that for all
$s\in(0,s_o)$
\begin{equation}
\label{eq:tiltCdH}
f_s(h_c)\le  (1+ \alpha)s^2 \,. 
\end{equation}
To control shifting by tilting, taking advantage of
\eqref{eq:tiltCdH}, we shall demonstrate that there exist constants
$\zeta>0$ and $\eta_o\in(0,\eta)$ such that
\begin{equation}
\int_\Omega (\omega_a-m_\sigma)\log Z_{n,h+\zeta(s-\sigma)}(\omega)\,\probd_{n,\sigma}[\dd\omega]
\ge\zeta\int_\Omega \Ex_{n,h+\zeta(s-\sigma),\omega}[X_a]\,\probd_{n,\sigma}[\dd\omega]
\label{eq:smoothing_goal}
\end{equation}
for all $n\in\N$, $h\in\Rl$, $s,\sigma\in(0,\eta_o)$, and
$a\in\{1,\ldots,n\}$. Since by direct computation we have
\begin{align}
  \nonumber
  &\partial_\sigma\int_\Omega\log Z_{n,h+\zeta(s-\sigma)}(\omega)\,\probd_{n,\sigma}[\dd\omega]\\
  \nonumber
      &~=\sum_{a=1}^n\Bigg\{\int_\Omega (\omega_a-m_\sigma)\log Z_{n,h+\zeta(s-\sigma)}(\omega)\,\probd_{n,\sigma}[\dd\omega]
      -\zeta\int_\Omega \Ex_{n,h+\zeta(s-\sigma),\omega}[X_a]\,\probd_{n,\sigma}[\dd\omega]\Bigg\}\,,
\end{align}
bound \eqref{eq:smoothing_goal} and an integration with respect to
$\sigma$ from 0 to $s$ imply that
\begin{equation*}
  \int_\Omega\log Z_{n,h+\zeta s}(\omega)\,\probd[\dd\omega]\le \int_\Omega\log Z_{n,h}(\omega)\,\probd_{n,s}[\dd\omega]
\end{equation*}
for every $n\in\N$, $h\in\Rl$, and $s\in(0,\eta_o)$.  Dividing the
latter by $n$ and then letting $n$ go to infinity, we deduce that
$f(h+\zeta s)\le f_s(h)$ for all $h\in\Rl$ and $s\in(0,\eta_o)$. It
therefore follows from \eqref{eq:tiltCdH} that if $h_c>-\infty$, then
$f(h_c+\delta)\le (\alpha+1)(\delta/\zeta)^2$ for all
$\delta\in(0,\min\{s_o,\eta_o\}/\zeta\})$, which proves the theorem
with $c:=(\alpha+1)/\zeta^2$.

\smallskip

\noindent\textit{The bound \eqref{eq:smoothing_goal}.}  The hypotheses
$\int_\Omega\omega_0\probd[\dd\omega]=0$ and
$\int_\Omega\omega_0^2\,\probd[\dd\omega]>0$ assure us that
$\int_\Omega\max\{0,\omega_0\}^{\!2}\,\probd[\dd\omega]>0$.  Set
\begin{equation*}
  \zeta:=\min\Bigg\{\frac{1}{8}\int_\Omega\max\{0,\omega_0\}^{\!2}\,\probd[\dd\omega],\frac{1}{4}\Bigg\}>0\,,
\end{equation*}
and for $s\in(0,\eta)$ define
\begin{equation*}
  m_s:=\frac{\int_\Omega \omega_0\ee^{s\omega_0}\probd[\dd\omega]}{\int_\Omega \ee^{s\omega_0}\probd[\dd\omega]}\,.
\end{equation*}
The dominated convergence theorem shows first that $\lim_{s\downarrow
  0}m_s=0$ and then that
\begin{equation*}
 \lim_{s\downarrow 0} \frac{\int_\Omega \max\{0,\omega_0-m_s\}^{\!2}\ee^{s\omega_0}\probd[\dd\omega]}
      {\int_\Omega \ee^{s\omega_0}\probd[\dd\omega]}=\int_\Omega \max\{0,\omega_0\}^{\!2}\,\probd[\dd\omega]\,,
\end{equation*}
so
\begin{equation}
 \frac{\int_\Omega \max\{0,\omega_0-m_s\}^{\!2}\ee^{s\omega_0}\probd[\dd\omega]}
      {\int_\Omega \ee^{s\omega_0}\probd[\dd\omega]}\ge\frac{1}{2}\int_\Omega \max\{0,\omega_0\}^{\!2}\,\probd[\dd\omega]\ge 4\zeta
      \label{eq:smoothing_zeta}
\end{equation}
for all $s\in(0,\eta_o)$ with a suitable choice of
$\eta_o\in(0,\eta)$. We claim that \eqref{eq:smoothing_goal} holds
with such constants $\zeta$ and $\eta_o$.

Fix $n\in\N$, $h\in\Rl$, $s,\sigma\in(0,\eta_o)$, and
$a\in\{1,\ldots,n\}$ and denote by ${}^a\omega$ the family of charges
where $\omega_a$ is replaced by $m_\sigma$: ${}^a\omega_b=m_\sigma$ if
$b=a$ and ${}^a\omega_b=\omega_b$ if $b\ne a$ . We have
\begin{align}
  \nonumber
  &\int_\Omega (\omega_a-m_\sigma)\log Z_{n,h+\zeta(s-\sigma)}(\omega)\,\probd_{n,\sigma}[\dd\omega]\\
  \nonumber
  &\qquad=\int_\Omega (\omega_a-m_\sigma)\Big[\log Z_{n,h+\zeta(s-\sigma)}(\omega)-\log Z_{n,h+\zeta(s-\sigma)}({}^a\omega)\Big]\probd_{n,\sigma}[\dd\omega]\\
  &\qquad\ge\int_\Omega \max\{0,\omega_a-m_\sigma\}\Big[\log Z_{n,h+\zeta(s-\sigma)}(\omega)-\log Z_{n,h+\zeta(s-\sigma)}({}^a\omega)\Big]\probd_{n,\sigma}[\dd\omega]\,,
\label{eq:smoothing_1}
\end{align}
where the equality is due to the fact that $\log
Z_{n,h+\zeta(s-\sigma)}({}^a\omega)$ is independent of $\omega_a$,
while the inequality is due to the fact that
$Z_{n,h+\zeta(s-\sigma)}(\omega)$ is increasing with respect to
$\omega_a$. Convexity gives for all
$\omega:=\{\omega_b\}_{b\in\N}\in\Omega$
\begin{equation*}
  \log Z_{n,h+\zeta(s-\sigma)}(\omega)-\log Z_{n,h+\zeta(s-\sigma)}({}^a\omega)\ge \Ex_{n,h+\zeta(s-\sigma),{}^a\omega}[X_a](\omega_a-m_\sigma)\,.
\end{equation*}
At the same time, the binary nature of $X_a$ allows us to write for
every $\omega:=\{\omega_b\}_{b\in\N}\in\Omega$
\begin{align}
  \nonumber
  \log Z_{n,h+\zeta(s-\sigma)}(\omega)-\log Z_{n,h+\zeta(s-\sigma)}({}^a\omega)&=\log\frac{\Ex\big[\ee^{\sum_{b=1}^n\{h+\zeta(s-\sigma)+\omega_b\}X_b}X_n\big]}
       {\Ex\big[\ee^{\sum_{b=1}^n\{h+\zeta(s-\sigma)+{}^a\omega_b\}X_b}X_n\big]}\\
       \nonumber
       &=\log\frac{\Ex\big[\ee^{\sum_{b=1}^n\{h+\zeta(s-\sigma)+\omega_b\}X_b}X_n\big]}
           {\Ex\big[\ee^{(m_\sigma-\omega_a)X_a+\sum_{b=1}^n\{h+\zeta(s-\sigma)+\omega_b\}X_b}X_n\big]}\\
           \nonumber
           &=-\log\Big\{1-(1-\ee^{m_\sigma-\omega_a})\Ex_{n,h+\zeta(s-\sigma),\omega}[X_a]\Big\}\\
           \nonumber
           &\ge (1-\ee^{m_\sigma-\omega_a})\Ex_{n,h+\zeta(s-\sigma),\omega}[X_a]\,.
\end{align}
Thus, we deduce from \eqref{eq:smoothing_1} that
\begin{align}
  \nonumber
  &\int_\Omega (\omega_a-m_\sigma)\log Z_{n,h+\zeta(s-\sigma)}(\omega)\,\probd_{n,\sigma}[\dd\omega]\\
  \nonumber
  &\qquad\ge \frac{1}{2} \int_\Omega \max\{0,\omega_a-m_\sigma\}^{\!2}\,\Ex_{n,h+\zeta(s-\sigma),{}^a\omega}[X_a]\,\probd_{n,\sigma}[\dd\omega]\\
  \nonumber
  &\qquad\quad +\frac{1}{2}\int_\Omega \max\{0,\omega_a-m_\sigma\}(1-\ee^{m_\sigma-\omega_a})\Ex_{n,h+\zeta(s-\sigma),\omega}[X_a]\,\probd_{n,\sigma}[\dd\omega]\\
  \nonumber
  &\qquad=\frac{1}{2}\frac{\int_\Omega\max\{0,\omega_0-m_\sigma\}^{\!2}\ee^{\sigma\omega_0}\probd[\dd\omega]}{\int_\Omega\ee^{\sigma\omega_0}\probd[\dd\omega]}
  \int_\Omega \Ex_{n,h+\zeta(s-\sigma),{}^a\omega}[X_a]\,\probd_{n,\sigma}[\dd\omega]\\
  &\qquad\quad +\frac{1}{2}\int_\Omega \max\{0,\omega_a-m_\sigma\}(1-\ee^{m_\sigma-\omega_a})\Ex_{n,h+\zeta(s-\sigma),\omega}[X_a]\,\probd_{n,\sigma}[\dd\omega]\,,
\label{eq:smoothing_2}
\end{align}
where the equality is due to the fact that the expectation
$\Ex_{n,h+\zeta(s-\sigma),{}^a\omega}[X_a]$ is independent of
$\omega_a$. Now we restore the charge $\omega_a$ in this expectation
by appealing again to the binary nature of $X_a$:
\begin{align}
  \nonumber
  \Ex_{n,h+\zeta(s-\sigma),{}^a\omega}[X_a]&=\frac{\Ex\big[X_a\ee^{\sum_{b=1}^n\{h+\zeta(s-\sigma)+{}^a\omega_b\}X_b}X_n\big]}
       {\Ex\big[\ee^{\sum_{b=1}^n\{h+\zeta(s-\sigma)+{}^a\omega_b\}X_b}X_n\big]}\\
       \nonumber
       &=\frac{\Ex\big[X_a\ee^{\sum_{b=1}^n\{h+\zeta(s-\sigma)+\omega_b\}X_b}X_n\big]}
           {\Ex\big[\ee^{-(\omega_a-m_\sigma)(X_a-1)+\sum_{b=1}^n\{h+\zeta(s-\sigma)+\omega_b\}X_b}X_n\big]}\\
           &\ge \ee^{-\max\{0,\omega_a-m_\sigma\}}\Ex_{n,h+\zeta(s-\sigma),\omega}[X_a]\,.
           \label{eq:smoothing_3}
\end{align}
Thus, combining \eqref{eq:smoothing_3} with \eqref{eq:smoothing_2} and
using \eqref{eq:smoothing_zeta}, as well as the fact that $1/2\ge
2\zeta$ by construction, we find
\begin{align}
  \nonumber
  &\int_\Omega (\omega_a-m_\sigma)\log Z_{n,h+\zeta(s-\sigma)}(\omega)\,\probd_{n,\sigma}[\dd\omega]\\
  \nonumber
  &~~\ge 2\zeta\int_\Omega \Big[\ee^{-\max\{0,\omega_a-m_\sigma\}}+\max\{0,\omega_a-m_\sigma\}(1-\ee^{m_\sigma-\omega_a})\Big]\Ex_{n,h+\zeta(s-\sigma),\omega}[X_a]\,\probd_{n,\sigma}[\dd\omega]\,.
\end{align}
This bound proves \eqref{eq:smoothing_goal} because
$\ee^{-\max\{0,z\}}+\max\{0,z\}(1-\ee^{-z})\ge 1/2$ for all $z\in\Rl$.
\end{proof}

\medskip

\section{Effect of disorder under contact number conditioning}
\label{sec:noBJ}

\subsection{Results from the companion paper \cite{cf:companionpaper}}
\label{sec:richiami}

The probability measure $\prob_{n,h,\omega}$ associated to the model
we consider is absolutely continuous with respect to the renewal law
$\prob$ and the relative density is the exponential of an energy term
which contains only one body potentials (see
\eqref{eq:model_binary}). From this it is not difficult to infer a
fundamental conditional independence property between consecutive
stretches of the polymer. In order to formulate such property, let
$\mathcal{M}_n$ be the set of functions $\phi:\{0,1\}^{n+1}\to\Cm$
such that $|\phi(x_0,\ldots,x_n)|\le 1$ for every
$x_0,\ldots,x_n\in\{0,1\}$ and denote by $\vartheta$ the left shift
that maps the sequence $\omega:=\{\omega_a\}_{a\in\N_0}$ to
$\vartheta\omega:=\{\omega_{a+1}\}_{a\in\N_0}$. The following
factorization, which implies the mentioned conditional independence,
is provided by \cite[Lemma 2.1]{cf:companionpaper}: for every
$h\in\Rl$, $\omega\in\Omega$, integers $0\le a\le n$, and functions
$\phi\in\mathcal{M}_a$ and $\psi\in\mathcal{M}_{n-a}$
\begin{align}
  \nonumber
  &\Ex_{n,h,\omega}\Big[\phi(X_0,\ldots,X_a)\psi(X_a,\ldots,X_n)\Big|X_a=1\Big]\\
  &\qquad=\Ex_{a,h,\omega}\big[\phi(X_0,\ldots,X_a)\big]\Ex_{n-a,h,\vartheta^a\omega}\big[\psi(X_0,\ldots,X_{n-a})\big]\,,
\label{eq:ci-fact1}
\end{align}
so that
\begin{align}
  \nonumber
  &\Ex_{n,h,\omega}\Big[\phi(X_0,\ldots,X_a)\psi(X_a,\ldots,X_n)\Big|X_a=1\Big]\\
  &\qquad=\Ex_{n,h,\omega}\big[\phi(X_0,\ldots,X_a)\big|X_a=1\big]\Ex_{n,h,\omega}\big[\psi(X_a,\ldots,X_n)\big|X_a=1\big]\,.
  \label{eq:ci-fact2}
\end{align}

\smallskip

When aiming at decorrelation estimates we will make use of the next
result, which is the first part of \cite[Lemma
  3.1]{cf:companionpaper}. Given two measurable complex functions
$\Phi$ and $\Psi$ over $(\mathcal{S},\mathfrak{S})$ we define their
covariance under the polymer measure $\prob_{n,h,\omega}$ as
\begin{equation*}
  \cov_{n,h,\omega}[\Phi,\Psi]:=\Ex_{n,h,\omega}\big[\Phi\overline{\Psi}\,\big]-\Ex_{n,h,\omega}\big[\Phi\big]\Ex_{n,h,\omega}\big[\,\overline{\Psi}\,\big]\,,
\end{equation*}
provided that all the expectations exist.  As usual, $\bar{\zeta}$ is
the complex conjugate of $\zeta\in\Cm$.  The variance
$\cov_{n,h,\omega}[\Phi,\Phi]$ is denoted by
$\var_{n,h,\omega}[\Phi]$.

\medskip

\begin{lemma}
  \label{lem:corr}
  For every $h\in\Rl$, $\omega:=\{\omega_c\}_{c\in\N_0}\in\Omega$,
  integers $0\le a\le b\le n$, and functions $\phi\in\mathcal{M}_a$
  and $\psi\in\mathcal{M}_{n-b}$
\begin{equation*}
\bigg|\cov_{n,h,\omega}\Big[\phi(X_0,\ldots,X_a),\psi(X_b,\ldots,X_n)\Big]\bigg|
\le 2\sum_{i=0}^{a-1}\sum_{j=b+1}^\infty\Ex_{j-i,h,\vartheta^i\omega}^{\otimes 2}\bigg[\prod_{k=1}^{j-i-1}(1-X_kX_k')\bigg]\,.
\end{equation*}
\end{lemma}

\medskip

Lemma~\ref{lem:corr} is demonstrated in \cite{cf:companionpaper} by
exploiting the trick of working with two \textit{replica} of the
system and then of using the conditional independence property
\eqref{eq:ci-fact2} in the event that the two replica meet in between
the sites $a$ and $b$. In fact, the lemma involves two independent
copies of the renewal process, $\{S_i\}_{i\in\N_0}$ and
$\{S_i'\}_{i\in\N_0}$ with associated contact indicators
$\{X_a\}_{a\in\N_0}$ and $\{X_a'\}_{a\in\N_0}$, defined on the
measurable space $(\mathcal{S}^2,\mathfrak{S}^{\otimes 2})$. The
product measure $\prob_{n,h,\omega}^{\otimes 2}$ on
$(\mathcal{S}^2,\mathfrak{S}^{\otimes 2})$ is the two-polymer measure
and $\Ex_{n,h,\omega}^{\otimes 2}$ is the corresponding expectation.

\smallskip

It will not come as a surprise then that the following estimate, which
is \cite[Lemma 3.2]{cf:companionpaper}, is of crucial importance. Here
and below the assumptions in force are Assumptions \ref{assump:p} and
\ref{assump:omega}.

\medskip

\begin{lemma}
  \label{lem:decay}
  For every closed set $H\subset(h_c,+\infty)$ there exist constants $\gamma>0$ and $G>0$ such
  that for all $n\in\N$
\begin{equation*}
  \Exd\Bigg[\sup_{h\in H}\Ex_{n,h,\cdot}^{\otimes 2}\bigg[\prod_{k=1}^{n-1}(1-X_kX_k')\bigg]\Bigg]\le G\,\ee^{-\gamma n}\,.
\end{equation*}
\end{lemma}

\medskip

One of the ingredients of Lemma \ref{lem:decay} is an
\textit{alternative free energy $\mu(h)$}, which can be defined by
 \begin{equation}
 \label{eq:mu-def}
 \mu(h):=-\limsup_{n \uparrow \infty} \frac 1n \log \Exd \Big[\prob_{n,h,\cdot}[T_1=n]\Big]\,.
 \end{equation}
Like $f(h)$, we have $\mu(h)>0$ for $h>h_c$ and $\mu(h)=0$ for $h\le
h_c$ (when $h_c>-\infty$) (see \cite[Proposition 1.7 and Corollary
  2.5]{cf:companionpaper}). The proof of these facts relies on a
version of the McDiarmid's inequality for functions of independent
random variables under subexponential conditions, which enables to
show the following concentration inequality for the finite-volume free
energy (see \cite[Theorem 2.4]{cf:companionpaper}).

\medskip

\begin{theorem}
  \label{th:concentration}
  There exists a constant $\kappa>0$ such that for every $n\in\N$, $h\in\Rl$,
  and $u\ge 0$
 \begin{equation*}
\probd\bigg[\Big|\log Z_{n,h}-\Exd\big[\log Z_{n,h}\big]\Big|>u\bigg]\le 2\ee^{-\frac{\kappa u^2}{n+u}}\,.
  \end{equation*}
\end{theorem}

\medskip

Lemmas \ref{lem:corr} and \ref{lem:decay} are at the basis of Theorem
\ref{th:Cinfty}. The latter implies that the contact number $L_n$
satisfies the following quenched CLT when $h>h_c$ (see \cite[part
  $(ii)$ of Theorem 1.4]{cf:companionpaper}).

\medskip

\begin{theorem}
  \label{th:CLT}
  The following properties hold for $\pae$: for every compact set $H\subset (h_c,+\infty)$
 \begin{equation*}
 \adjustlimits\limsmallspace_{n\uparrow\infty}\sup_{h\in H}\,\sup_{u\in\Rl}\,\Bigg|\prob_{n,h,\omega}\bigg[\frac{L_n-\Ex_{n,h,\omega}[L_n]}{\sqrt{n v(h)}}\le u\bigg]-
     \frac{1}{\sqrt{2\pi}}\int_{-\infty}^u\ee^{-\frac{1}{2}z^2}\dd z\Bigg|=0
\end{equation*}
with $v(h):=\partial_h^2 f(h)>0$.
\end{theorem}

\medskip

\subsection{The local CLT for the contact number}
\label{sec:localCLT}

The proof of the quenched local CLT for the contact number requires
that we investigate the characteristic function of $L_n$ with respect
to the polymer measure $\prob_{n,h,\omega}$.  Put for brevity
$\#_n:=\lfloor(n-1)/2\rfloor$ and
 \begin{equation*}
    J_{n,h,\omega}:=\sum_{k=1}^{\#_n}\frac{2\ee^{h+\omega_{2k}}p(1)^2p(2)}{[p(2)+\ee^{h+\omega_{2k}}p(1)^2]^2}X_{2k-1}X_{2k+1}\,.
  \end{equation*}
 The following lemma, which is partly inspired by \cite[Lemmas 2.2 and
   2.3]{campanino1979}, deals with the mentioned characteristic
 function and is the first step towards Theorem \ref{th:LCLT}.
 
 \medskip
 
 \begin{lemma}
   \label{lem:charac}
   For every integer $n\ge 3$, $h\in\Rl$, $\omega\in\Omega$, and
   $z\in[-\pi,\pi]$
 \begin{equation*}
   \Big|\Ex_{n,h,\omega}\big[\ee^{\mathrm{i}z L_n}\big]\Big|\le\Ex_{n,h,\omega}\Big[\ee^{-\frac{z^2}{\pi^2}J_{n,h,\omega}}\Big]\le
   \ee^{-\frac{z^2}{2\pi^2}\Ex_{n,h,\omega}[J_{n,h,\omega}]}+4\frac{\var_{n,h,\omega}[J_{n,h,\omega}]}{\Ex_{n,h,\omega}[J_{n,h,\omega}]^2}\,.
\end{equation*}
 \end{lemma}
 
 \medskip

\begin{proof}[Proof of Lemma \ref{lem:charac}]
Fix an integer $n\ge 3$, $h\in\Rl$, $\omega\in\Omega$, and
$z\in[-\pi,\pi]$. The second inequality is due to the first and to
Chebyshev's inequality, which entail
\begin{align}
  \nonumber
  \Ex_{n,h,\omega}\Big[\ee^{-\frac{z^2}{\pi^2}J_{n,h,\omega}}\Big]&\le \ee^{-\frac{z^2}{2\pi^2}\Ex_{n,h,\omega}[J_{n,h,\omega}]}+
  \prob_{n,h,\omega}\bigg[J_{n,h,\omega}\le \frac{1}{2}\Ex_{n,h,\omega}[J_{n,h,\omega}]\bigg]\\
  \nonumber
  &\le \ee^{-\frac{z^2}{2\pi^2}\Ex_{n,h,\omega}[J_{n,h,\omega}]}+4\frac{\var_{n,h,\omega}[J_{n,h,\omega}]}{\Ex_{n,h,\omega}[J_{n,h,\omega}]^2}\,.
\end{align}

\smallskip

Regarding the first inequality, denoting by $\mathfrak{O}$ the
$\sigma$-algebra generated by the odd binary variables
$X_1,X_3,\ldots$, we can state that
\begin{equation*}
  \Big|\Ex_{n,h,\omega}\big[\ee^{\mathrm{i}z L_n}\big]\Big|\le
  \Ex_{n,h,\omega}\bigg[\bigg|\Ex_{n,h,\omega}\Big[\ee^{\mathrm{i}z\sum_{k=1}^{\#_n}X_{2k}}\Big|\mathfrak{O}\Big]\bigg|\bigg]\,.
\end{equation*}
Therefore, it suffices to show that
\begin{equation}
  \label{charac_1}
  \bigg|\Ex_{n,h,\omega}\Big[\ee^{\mathrm{i}z\sum_{k=1}^{\#_n}X_{2k}}\Big|\mathfrak{O}\Big]\bigg|\le \ee^{-\frac{z^2}{\pi^2}J_{n,h,\omega}}\,.
\end{equation}
This bound is trivial if at most one variable among
$X_1,X_3,\ldots,X_{2\#_n+1}$ takes value 1 since $J_{n,h,\omega}=0$ in
this case. Then, suppose that at least two of these variables take
value 1 and let $ 2l_1+1<\cdots<2l_r+1$ with $r\ge 2$ be the sites
where the values 1 are attained. Expliciting $J_{n,h,\omega}$ we see
that \eqref{charac_1} is tantamount to the following bound to be
proved:
\begin{align}
   \nonumber
  &\bigg|\Ex_{n,h,\omega}\Big[\ee^{\mathrm{i}z\sum_{k=1}^{\#_n}X_{2k}}\Big|\mathfrak{O}\Big]\bigg|\\
   &\qquad\le\prod_{s=2}^r\Bigg(\mathds{1}_{\{l_s\ne l_{s-1}+1\}}+\mathds{1}_{\{l_s=l_{s-1}+1\}}
   \exp\bigg\{-\frac{2z^2\ee^{h+\omega_{2l_s}}p(1)^2p(2)}{\pi^2[p(2)+\ee^{h+\omega_{2l_s}}p(1)^2]^2}\bigg\}\Bigg)\,.
  \label{charac_2}
\end{align}

Let us verify \eqref{charac_2}.  Because of \eqref{eq:ci-fact2},
conditioning to the odd binary variables makes the sums
$\sum_{k=1}^{l_1}X_{2k},\sum_{k=l_1+1}^{l_2}X_{2k},\ldots,\sum_{k=l_{r-1}+1}^{l_r}X_{2k},\sum_{k=l_r+1}^{\#_n}X_{2k}$
independent, so that
\begin{align}
  \nonumber
  \bigg|\Ex_{n,h,\omega}\Big[\ee^{\mathrm{i}z\sum_{k=1}^{\#_n}X_{2k}}\Big|\mathfrak{O}\Big]\bigg|
  &= \bigg|\Ex_{n,h,\omega}\Big[\ee^{\mathrm{i}z\sum_{k=1}^{l_1}X_{2k}}\Big|\mathfrak{O}\Big]\bigg|\\
  \nonumber
  &\quad\times\prod_{s=2}^r\bigg|\Ex_{n,h,\omega}\Big[\ee^{\mathrm{i}z\sum_{k=l_{s-1}+1}^{l_s}X_{2k}}\Big|\mathfrak{O}\Big]\bigg|\\
  \nonumber
  &\quad\quad\times\bigg|\Ex_{n,h,\omega}\Big[\ee^{\mathrm{i}z\sum_{k=l_r+1}^{\#_n}X_{2k}}\Big|\mathfrak{O}\Big]\bigg|\\
  &\le \prod_{s=2}^r\Bigg\{\mathds{1}_{\{l_s\ne l_{s-1}+1\}}+\mathds{1}_{\{l_s=l_{s-1}+1\}}\bigg|\Ex_{n,h,\omega}\Big[\ee^{\mathrm{i}z X_{2l_s}}\Big|\mathfrak{O}\Big]\bigg|\Bigg\}\,.
\label{charac_3}
\end{align}
Under the hypothesis $l_s=l_{s-1}+1$ two applications of
\eqref{eq:ci-fact1} yield
\begin{equation*}
  \Ex_{n,h,\omega}\Big[\ee^{\mathrm{i}z X_{2l_s}}\Big|\mathfrak{O}\Big]=
  \Ex_{2,h,\vartheta^{2l_s-1}\omega}\big[\ee^{\mathrm{i}z X_1}\big]=\frac{p(2)+\ee^{\mathrm{i}z+h+\omega_{2l_s}}p(1)^2}{p(2)+\ee^{h+\omega_{2l_s}}p(1)^2}\,,
\end{equation*}
and the bounds $\cos z\le 1-\frac{2z^2}{\pi^2}$, as $z\in[-\pi,\pi]$,
and $\sqrt{1-\zeta}\le \ee^{-\frac{\zeta}{2}}$ for $\zeta\in[0,1]$
give
\begin{align}
  \nonumber
  \bigg|\Ex_{n,h,\omega}\Big[\ee^{\mathrm{i}z X_{2l_s}}\Big|\mathfrak{O}\Big]\bigg|&=\frac{\sqrt{p(2)^2+\ee^{2h+2\omega_{2l_s}}p(1)^4
      +2\ee^{h+\omega_{2l_s}}p(1)^2p(2)\cos z}}{p(2)+\ee^{h+\omega_{2l_s}}p(1)^2}\\
  \nonumber
  &\le\sqrt{1-\frac{4z^2\ee^{h+\omega_{2l_s}}p(1)^2p(2)}{\pi^2[p(2)+\ee^{h+\omega_{2l_s}}p(1)^2]^2}}\\
  &\le \exp\left\{-\frac{2z^2\ee^{h+\omega_{2l_s}}p(1)^2p(2)}{\pi^2[p(2)+\ee^{h+\omega_{2l_s}}p(1)^2]^2}\right\}\,.
  \label{charac_4}
\end{align}
We find (\ref{charac_2}) by combining (\ref{charac_3}) with
(\ref{charac_4}).
\end{proof}

\medskip

The next step is to investigate the mean and the variance of
$J_{n,h,\omega}$.

\medskip

\begin{lemma}
  \label{lem:bounds_J}
  The following properties hold for $\pae$:
  \vspace{0.1cm}
  \begin{enumerate}[itemsep=0.3em,label=({\roman*})]
    \item $\displaystyle{\liminf_{n\uparrow\infty}\inf_{h\in H}\frac{1}{n}\Ex_{n,h,\omega}[J_{n,h,\omega}]>0}$ for all $H\subset(h_c,+\infty)$ compact;
    \item $\displaystyle{\limsup_{n\uparrow\infty}\sup_{h\in H}\frac{1}{n}\var_{n,h,\omega}[J_{n,h,\omega}]<+\infty}$
      for all $H\subset(h_c,+\infty)$ compact.
  \end{enumerate}
\end{lemma}

\medskip

\begin{proof}[Proof of Lemma \ref{lem:bounds_J}]
The properties of slowly varying functions (see \cite[Theorem 1.2.1
   and Proposition 1.3.6]{bingham1989}) imply that there exists a
 constant $\xi>0$ such that for all $t,\tau\in\N$
\begin{equation}
  \frac{p(t+\tau)}{p(t)p(\tau)}\le\xi \min\{t^\xi,\tau^\xi\}\,.
  \label{eq:xi_def}
\end{equation}
The constant $\xi$ enters a technical lower bound on the probability
that there is a contact at a given site. In fact, for all $h\in\Rl$,
$\omega:=\{\omega_a\}_{a\in\N_0}\in\Omega$, and integers $0\le a\le n$
we have (see \cite[Lemma 2.2 and Remark 2.3]{cf:companionpaper})
\begin{equation}
  \Ex_{n,h,\omega}[X_a]\ge\frac{1}{1+\xi\ee^{-h-\omega_a}\min\{a^\xi,(n-a)^\xi\}}\,.
\label{eq:lower_bound_pcontact}
\end{equation}
Such lower bound is at the basis of the following estimate
that we need to prove the lemma: for every integer $n\ge 3$,
$h\in\Rl$, $\omega:=\{\omega_a\}_{a\in\N_0}\in\Omega$, and
$k\in\{1,\ldots,\#_n\}$ with $\#_n:=\lfloor(n-1)/2\rfloor$
\begin{align}
  \nonumber
  &\Ex_{n,h,\omega}[X_{2k-1}X_{2k+1}]\\
  &\qquad\ge \frac{\ee^{-2|h|}}{2}\frac{1}{1+\xi 2^\xi \ee^{-\omega_{2k-1}}}\frac{1}{1+\xi\ee^{-\omega_{2k+1}}}\Big(\Ex_{n,h,\omega}[X_{2k}]+\Ex_{n,h,\omega}[X_{2k+1}]\Big)\,.
  \label{lower_bound_XX}
\end{align}

In order to prove \eqref{lower_bound_XX} we appeal to
\eqref{eq:ci-fact1} to state that for a given integer $n\ge 3$,
$h\in\Rl$, $\omega:=\{\omega_a\}_{a\in\N_0}\in\Omega$, and
$k\in\{1,\ldots,\#_n\}$
\begin{align}
  \nonumber
  \Ex_{n,h,\omega}[X_{2k-1}X_{2k+1}]&=\Ex_{n,h,\omega}[X_{2k-1}|X_{2k+1}=1]\Ex_{n,h,\omega}[X_{2k+1}]\\
  \nonumber
  &=\Ex_{2k+1,h,\omega}[X_{2k-1}]\Ex_{n,h,\omega}[X_{2k+1}]
\end{align}
and
\begin{align}
  \nonumber
  \Ex_{n,h,\omega}[X_{2k+1}]\ge \Ex_{n,h,\omega}[X_{2k}X_{2k+1}]&=\Ex_{n,h,\omega}[X_{2k+1}|X_{2k}=1]\Ex_{n,h,\omega}[X_{2k}]\\
  \nonumber
  &=\Ex_{n-2k,h,\vartheta^{2k}\omega}[X_1]\Ex_{n,h,\omega}[X_{2k}]\,,
\end{align}
so that
\begin{align}
  \nonumber
  \Ex_{n,h,\omega}[X_{2k-1}X_{2k+1}]&\ge \frac{1}{2}\Ex_{2k+1,h,\omega}[X_{2k-1}]\Ex_{n,h,\omega}[X_{2k+1}]\\
  \nonumber
  &\quad +\frac{1}{2}\Ex_{2k+1,h,\omega}[X_{2k-1}]\Ex_{n-2k,h,\vartheta^{2k}\omega}[X_1]\Ex_{n,h,\omega}[X_{2k}]\\
  \nonumber
  &\ge \frac{1}{2}\Ex_{2k+1,h,\omega}[X_{2k-1}]\Ex_{n-2k,h,\vartheta^{2k}\omega}[X_1]\Big(\Ex_{n,h,\omega}[X_{2k}]+\Ex_{n,h,\omega}[X_{2k+1}]\Big)\,.
\end{align}
This gives \eqref{lower_bound_XX} because
\eqref{eq:lower_bound_pcontact} shows that
 \begin{equation*}
 \Ex_{2k+1,h,\omega}[X_{2k-1}]\ge \frac{1}{1+\xi 2^\xi \ee^{-h-\omega_{2k-1}}}\ge \frac{\ee^{-|h|}}{1+\xi 2^\xi \ee^{-\omega_{2k-1}}}
\end{equation*}
and
\begin{equation*}
  \Ex_{n-2k,h,\vartheta^{2k}\omega}[X_1]\ge \frac{1}{1+\xi\ee^{-h-\omega_{2k+1}}}\ge \frac{\ee^{-|h|}}{1+\xi\ee^{-\omega_{2k+1}}}\,.
\end{equation*}

\smallskip

\noindent\textit{Part (i).}  Let $\Lambda$ be the random variable that
maps $\omega:=\{\omega_a\}_{a\in\N_0}\in\Omega$ to
\begin{equation*}
\Lambda(\omega):=\frac{\ee^{\omega_1}p(1)^2p(2)}{[p(2)+\ee^{\omega_1}p(1)^2]^2}\frac{1}{1+\xi 2^\xi \ee^{-\omega_0}}\frac{1}{1+\xi\ee^{-\omega_2}}\,,
\end{equation*}
$\xi$ being the constant introduced in \eqref{eq:xi_def}.  Note the
$\Lambda$ is independent of the parameter $h$. Bounds
\eqref{lower_bound_XX} gives for every integer $n\ge 3$, $h\in\Rl$,
and $\omega:=\{\omega_a\}_{a\in\N_0}\in\Omega$
\begin{align}
  \nonumber
  \Ex_{n,h,\omega}[J_{n,h,\omega}]&=\sum_{k=1}^{\#_n}\frac{2\ee^{h+\omega_{2k}}p(1)^2p(2)}{[p(2)+\ee^{h+\omega_{2k}}p(1)^2]^2}\,\Ex_{n,h,\omega}[X_{2k-1}X_{2k+1}]\\
  \nonumber
  &\ge\ee^{-3|h|}\sum_{k=1}^{\#_n}\frac{2\ee^{\omega_{2k}}p(1)^2p(2)}{[p(2)+\ee^{\omega_{2k}}p(1)^2]^2}\,\Ex_{n,h,\omega}[X_{2k-1}X_{2k+1}]\\
  &\ge \ee^{-5|h|}\sum_{k=1}^{\#_n}\Lambda(\vartheta^{2k-1}\omega)\Big(\Ex_{n,h,\omega}[X_{2k}]+\Ex_{n,h,\omega}[X_{2k+1}]\Big)\,.
  \label{eq:bounds_J_0}
\end{align}
This inequality allows us to demonstrate part $(i)$ of the lemma as
follows. According to Theorem \ref{th:Cinfty} and Birkhoff's ergodic
theorem there exists a set $\Omega_o\in\mathcal{F}$ with
$\probd[\Omega_o]=1$ such that the following properties hold for all
$\omega\in\Omega_o$:
\vspace{0.1cm}
  \begin{enumerate}[itemsep=0.3em]
    \item $\displaystyle{\adjustlimits\lim_{n\uparrow\infty}\sup_{h\in H}\bigg|\Ex_{n,h,\omega}\bigg[\frac{L_n}{n}\bigg]-\rho(h)\bigg|=0}$ for every compact set $H\subset(h_c,+\infty)$;
    \item $\displaystyle{\lim_{n\uparrow\infty}\frac{1}{n}\sum_{i=0}^{n-1}\mathds{1}_{\{\Lambda(\vartheta^i\omega)\le 1/s\}}=\probd[\Lambda\le 1/s]}$ for every $s\in\N$.
  \end{enumerate}
Fix $\omega\in\Omega_o$ and $H\subset(h_c,+\infty)$ compact and put
$h_o:=\sup\{|h|:h\in H\}<+\infty$. Note that $\inf_{h\in H}\rho(h)>0$,
as the function $\rho:=\partial_hf$ is continuous and strictly
positive on $(h_c,+\infty)$, and that
$\lim_{s\uparrow\infty}\probd[\Lambda\le 1/s]=0$ by 
monotone continuity of probabilities. Then, choose $s\in\N$ so large that
$\probd[\Lambda\le 1/s]\le (1/3)\inf_{h\in H}\rho(h)$.  Bound
\eqref{eq:bounds_J_0} entails that for every integer $n\ge 3$ and
$h\in H$
\begin{align}
  \nonumber
  \Ex_{n,h,\omega}[J_{n,h,\omega}]&\ge \frac{\ee^{-5h_o}}{s}\sum_{k=1}^{\#_n}\mathds{1}_{\{\Lambda(\vartheta^{2k-1}\omega)>1/s\}}
  \Big(\Ex_{n,h,\omega}[X_{2k}]+\Ex_{n,h,\omega}[X_{2k+1}]\Big)\\
  \nonumber
  &\ge\frac{\ee^{-5h_o}}{s}\sum_{k=1}^{\#_n}\Big(\Ex_{n,h,\omega}[X_{2k}]+\Ex_{n,h,\omega}[X_{2k+1}]\Big)
  -\frac{2\ee^{-5h_o}}{s}\sum_{k=1}^{\#_n}\mathds{1}_{\{\Lambda(\vartheta^{2k-1}\omega)\le 1/s\}}\\
  \nonumber
  &\ge\frac{\ee^{-5h_o}}{s}\Ex_{n,h,\omega}[L_n]-\frac{2}{s}-\frac{2\ee^{-5h_o}}{s}\sum_{i=0}^{n-1}\mathds{1}_{\{\Lambda(\vartheta^i\omega)\le 1/s\}}\,.
\end{align}
In this way, Properties $(1)$ and $(2)$ give
\begin{align}
  \nonumber
  \liminf_{n\uparrow\infty}\inf_{h\in H}\frac{1}{n}\Ex_{n,h,\omega}[J_{n,h,\omega}]&\ge\frac{\ee^{-5h_o}}{s}\inf_{h\in H}\rho(h)-\frac{2\ee^{-5h_o}}{s}\probd[\Lambda\le 1/s]\\
  \nonumber
  &\ge\frac{\ee^{-5h_o}}{3s}\inf_{h\in H}\rho(h)>0\,.
\end{align}

\smallskip

\noindent\textit{Part (ii).} To begin with, for $s\in\N$ put $h_s:=-s$
if $h_c=-\infty$ and $h_s:=h_c+1/s$ if $h_c>-\infty$ and note that
$h_s>h_c$ and $\lim_{s\uparrow\infty}h_s=h_c$.  Then, consider the
random variables
\begin{equation*}
  \Lambda_s:=\sum_{j\in\N} j^2 \sup_{h\in[h_s,+\infty)}\Ex_{j,h,\cdot}^{\otimes 2}\bigg[\prod_{k=1}^{j-1}(1-X_kX_k')\bigg]\,,
\end{equation*}
which possess finite expectation by Lemma \ref{lem:decay}.  Birkhoff's
ergodic theorem ensures us that there exists $\Omega_o\in\mathcal{F}$
with $\probd[\Omega_o]=1$ such that for every $\omega\in\Omega_o$ and
$s\in\N$
\begin{equation}
  \lim_{n\uparrow\infty}\frac{1}{n}\sum_{i=0}^{n-1}\Lambda_s(\vartheta^i\omega)=\Exd[\Lambda_s]\,.
  \label{eq:var_J_finite}
\end{equation}
Fix $\omega\in\Omega_o$ and $H\subset(h_c,+\infty)$ compact and pick
$s$ so large that $[h_s,+\infty)\supset H$. Using that
  $2\ee^{h+\omega_{2k}}p(1)^2p(2)/
  [p(2)+\ee^{h+\omega_{2k}}p(1)^2]^2\le 1$ we can write
\begin{equation*}
  \var_{n,h,\omega}[J_{n,h,\omega}] \le
  2\sum_{a=1}^{\#_n}\sum_{b=a}^{\#_n}\Big|\cov_{n,h,\omega}\big[X_{2a-1}X_{2a+1},X_{2b-1}X_{2b+1}\big]\Big|\,.
\end{equation*}
Then, Lemma \ref{lem:corr} gives for all integers $n\ge 3$ and $h\in
H$
\begin{align}
  \nonumber
  \var_{n,h,\omega}[J_{n,h,\omega}]&\le 2\#_n+4\sum_{a=1}^{\#_n-1}\sum_{b=a+1}^{\#_n}\sum_{i=0}^{2a}\sum_{j=2b}^\infty\Ex_{j-i,h,\vartheta^i\omega}^{\otimes 2}\bigg[\prod_{k=1}^{j-i-1}(1-X_kX_k')\bigg]\\
  \nonumber
  &\le n+\sum_{i=0}^{n-1}\sum_{j\in\N} j^2\Ex_{j,h,\vartheta^i\omega}^{\otimes 2}\bigg[\prod_{k=1}^{j-1}(1-X_kX_k')\bigg]
  \le n+\sum_{i=0}^{n-1}\Lambda_s(\vartheta^i\omega)\,.
\end{align}
It therefore follows from \eqref{eq:var_J_finite} that
\begin{equation*}
  \limsup_{n\uparrow\infty}\sup_{h\in H}\frac{\var_{n,h,\omega}[J_{n,h,\omega}]}{n}\le 1+\Exd[\Lambda_s]<+\infty\,.
  \qedhere
  \end{equation*}
\end{proof}

\medskip

We are now able to prove Theorem \ref{th:LCLT}.

\medskip

\begin{proof}[Proof of Theorem \ref{th:LCLT}]
We shall show that for every $n\in\N$, $h\in\Rl$, $\omega\in\Omega$,
and positive number $\lambda<\pi\sqrt{v(h)n}$
\begin{align}
  \nonumber
  &\sup_{l\in\N_0}\Bigg|\sqrt{2\pi v(h)n}\,\prob_{n,h,\omega}[L_n=l]-\exp\bigg\{\!-\frac{(l-\Ex_{n,h,\omega}[L_n])^2}{2v(h)n}\bigg\}\Bigg|\\
  \nonumber
  &\qquad\le \frac{8\lambda+2\lambda^3}{\sqrt{2\pi}}\sup_{u\in\Rl} \Bigg|\prob_{n,h,\omega}\bigg[\frac{L_n-\Ex_{n,h,\omega}[L_n]}{\sqrt{v(h)n}}\le u\bigg]
  -\frac{1}{\sqrt{2\pi}}\int_{-\infty}^u\ee^{-\frac{1}{2}z^2}\dd z\Bigg|\\
  &\qquad\quad
  +\sqrt{\frac{\pi}{2}}\int_{\lambda}^{\pi \sqrt{v(h)n}}\bigg|\Ex_{n,h,\omega}\Big[\ee^{\mathrm{i} z\frac{L_n}{\sqrt{v(h)n}}}\Big]\bigg|\,\dd z
  +\bigg(\sqrt{\frac{\pi}{2}}+\frac{8\lambda}{\pi}\bigg)\int_\lambda^{+\infty}\ee^{-\frac{1}{2}z^2}\dd z\,.
  \label{eq:LCLT_0}
\end{align}
Bound \eqref{eq:LCLT_0} proves the theorem as follows.  Lemma
\ref{lem:bounds_J} and Theorem \ref{th:CLT} assure us that there
exists a set $\Omega_o\in\mathcal{F}$ with $\probd[\Omega_o]=1$ such
that the following properties hold for every $\omega\in\Omega_o$ and
$H\subset(h_c,+\infty)$ compact:
  \vspace{0.1cm}
  \begin{enumerate}[itemsep=0.3em]
    \item $\displaystyle{\liminf_{n\uparrow\infty}\inf_{h\in H}\frac{1}{n}\Ex_{n,h,\omega}[J_{n,h,\omega}]>0\,;}$
    \item $\displaystyle{\limsup_{n\uparrow\infty}\sup_{h\in H}\frac{1}{n}\var_{n,h,\omega}[J_{n,h,\omega}]<+\infty\,;}$
    \item $\displaystyle{\adjustlimits\limsmallspace_{n\uparrow\infty}\sup_{h\in H}\,\sup_{u\in\Rl}\,\Bigg|\prob_{n,h,\omega}
      \bigg[\frac{L_n-\Ex_{n,h,\omega}[L_n]}{\sqrt{v(h)n}}\le u\bigg]-\frac{1}{\sqrt{2\pi}}\int_{-\infty}^u\ee^{-\frac{1}{2}z^2}\dd z\Bigg|=0}\,.$
      \end{enumerate}
Pick $\omega\in\Omega_o$ and a compact set $H\subset(h_c,+\infty)$.
Properties $(1)$ and $(2)$ show that $\inf_{h\in
  H}\Ex_{n,h,\omega}[J_{n,h,\omega}]\ge c n$ and $\sup_{h\in
  H}\var_{n,h,\omega}[J_{n,h,\omega}]\le n/c$ for every integer
$n>n_o$ with some numbers $c>0$ and $n_o\ge 3$.  The continuity and
positivity of $v:=\partial_h^2f$ on $(h_c,+\infty)$ tell us that $c$
can be chosen so small to also have $\inf_{h\in H}v(h)\ge c$ and
$\sup_{h\in H}v(h)\le 1/c$. Then, Lemma \ref{lem:charac} implies for
an integer $n>n_o$, $h\in H$, and
$z\in[-\pi\sqrt{v(h)n},\pi\sqrt{v(h)n}]$
 \begin{equation*}
   \bigg|\Ex_{n,h,\omega}\Big[\ee^{\mathrm{i}\frac{z}{\sqrt{v(h)n}}L_n}\Big]\bigg|\le \ee^{-\frac{c}{2\pi^2v(h)}z^2}+\frac{4}{c^3n}\le \ee^{-\frac{c^2}{2\pi^2}z^2}+\frac{4}{c^3n}\,.
 \end{equation*}
It therefore follows from \eqref{eq:LCLT_0} that for all integers
$n>n_o$ and positive numbers $\lambda<\pi\sqrt{cn}$
\begin{align}
  \nonumber
  &\sup_{h\in H}\,\sup_{l\in\N_0}\Bigg|\sqrt{2\pi v(h)n}\,\prob_{n,h,\omega}[L_n=l]-\exp\bigg\{\!-\frac{(l-\Ex_{n,h,\omega}[L_n])^2}{2v(h)n}\bigg\}\Bigg|\\
  \nonumber
  &\qquad\le \frac{8\lambda+2\lambda^3}{\sqrt{2\pi}}\sup_{h\in H}\,
  \sup_{u\in\Rl} \Bigg|\prob_{n,h,\omega}\bigg[\frac{L_n-\Ex_{n,h,\omega}[L_n]}{\sqrt{v(h)n}}\le u\bigg]-\frac{1}{\sqrt{2\pi}}\int_{-\infty}^u\ee^{-\frac{1}{2}z^2}\dd z\Bigg|\\
  \nonumber
  &\qquad\quad
  +\sqrt{\frac{\pi}{2}}\int_{\lambda}^{+\infty}\ee^{-\frac{c^2}{2\pi^2}z^2}\dd z+\sqrt{\frac{8\pi^3}{c^7n}}
  +\bigg(\sqrt{\frac{\pi}{2}}+\frac{8\lambda}{\pi}\bigg)\int_\lambda^{+\infty}\ee^{-\frac{1}{2}z^2}\dd z\,,
\end{align}
which thanks to Property $(3)$ gives for every $\lambda>0$
\begin{align}
  \nonumber
  &\limsup_{n\uparrow\infty}\sup_{h\in H}\,\sup_{l\in\N_0}\Bigg|\sqrt{2\pi v(h)n}\,\prob_{n,h,\omega}[L_n=l]-\exp\bigg\{\!-\frac{(l-\Ex_{n,h,\omega}[L_n])^2}{2v(h)n}\bigg\}\Bigg|\\
  \nonumber
  &\qquad\le \sqrt{\frac{\pi}{2}}\int_{\lambda}^{+\infty}\ee^{-\frac{c^2}{2\pi^2}z^2}\dd z
  +\bigg(\sqrt{\frac{\pi}{2}}+\frac{8\lambda}{\pi}\bigg)\int_\lambda^{+\infty}\ee^{-\frac{1}{2}z^2}\dd z\,.
\end{align}
Theorem~\ref{th:LCLT} follows from here by letting $\lambda$ go to
infinity.

\smallskip

\noindent\textit{The bound \eqref{eq:LCLT_0}.}  Fix $n\in\N$, $h\in\Rl$,
$\omega\in\Omega$, and a positive number
$\lambda<\pi\sqrt{v(h)n}$. The identities
\begin{align}
  \nonumber
  \sqrt{2\pi v(h)n}\,\prob_{n,h,\omega}[L_n=l]&=\sqrt{\frac{v(h)n}{2\pi}}\int_{-\pi}^\pi \Ex_{n,h,\omega}\big[\ee^{\mathrm{i}z L_n}\big]\ee^{-\mathrm{i}z l}\dd z\\
  \nonumber
  &=\frac{1}{\sqrt{2\pi}}\int_{-\pi\sqrt{v(h)n}}^{\pi\sqrt{v(h)n}} \Ex_{n,h,\omega}\bigg[\ee^{\mathrm{i}z \frac{L_n-\Ex_{n,h,\omega}[L_n]}{\sqrt{v(h)n}}}\bigg]
  \ee^{\mathrm{i}z \frac{\Ex_{n,h,\omega}[L_n]-l}{\sqrt{v(h)n}}}\dd z
\end{align}
and
\begin{equation*}
  \exp\bigg\{\!-\frac{(l-\Ex_{n,h,\omega}[L_n])^2}{2v(h)n}\bigg\}=\frac{1}{\sqrt{2\pi}}\int_{-\infty}^{+\infty}
  \ee^{\mathrm{i}z\frac{\Ex_{n,h,\omega}[L_n]-l}{\sqrt{v(h)n}} -\frac{1}{2}z^2} \dd z\,,
\end{equation*}
which are valid for every $l\in\N_0$, yield
\begin{align}
  \nonumber
  &\sup_{l\in\N_0}\Bigg|\sqrt{2\pi v(h)n}\,\prob_{n,h,\omega}[L_n=l]-\exp\bigg\{\!-\frac{(l-\Ex_{n,h,\omega}[L_n])^2}{2v(h)n}\bigg\}\Bigg|\\
  \nonumber
  &\qquad\le \frac{1}{\sqrt{2\pi}}\int_{-\pi\sqrt{v(h)n}}^{\pi\sqrt{v(h)n}}\Bigg|\Ex_{n,h,\omega}\bigg[\ee^{\mathrm{i}z\frac{L_n-\Ex_{n,h,\omega}[L_n]}{\sqrt{v(h)n}}}\bigg]
  -\ee^{-\frac{1}{2}z^2}\Bigg|\,\dd z+\sqrt{\frac{\pi}{2}}\int_{\pi\sqrt{v(h)n}}^{+\infty}\ee^{-\frac{1}{2}z^2}\dd z\\
  \nonumber
  &\qquad\le \frac{1}{\sqrt{2\pi}}\int_{-\lambda}^\lambda\Bigg|\Ex_{n,h,\omega}\bigg[\ee^{\mathrm{i}z\frac{L_n-\Ex_{n,h,\omega}[L_n]}{\sqrt{v(h)n}}}\bigg]
  -\ee^{-\frac{1}{2}z^2}\Bigg|\,\dd z\\
  \nonumber
  &\qquad\quad
  +\sqrt{\frac{\pi}{2}}\int_\lambda^{\pi \sqrt{v(h)n}}\bigg|\Ex_{n,h,\omega}\Big[\ee^{\mathrm{i} z\frac{L_n}{\sqrt{v(h)n}}}\Big]\bigg|\,\dd z
  +\sqrt{\frac{\pi}{2}}\int_\lambda^{+\infty}\ee^{-\frac{1}{2}z^2}\dd z\,.
\end{align}
Setting
$\mathcal{N}(u):=(1/\sqrt{2\pi})\int_{-\infty}^u\ee^{-z^2/2}\dd z$ for
brevity, we obtain \eqref{eq:LCLT_0} from here if we show that for all
$z\in\Rl$
\begin{align}
  \nonumber
  &\Bigg|\Ex_{n,h,\omega}\bigg[\ee^{\mathrm{i}z\frac{L_n-\Ex_{n,h,\omega}[L_n]}{\sqrt{v(h)n}}}\bigg]-\ee^{-\frac{1}{2}z^2}\Bigg|\\
  \nonumber
  &\qquad\le (4+2|z|\lambda)\sup_{u\in\Rl} \Bigg|\prob_{n,h,\omega}\bigg[\frac{L_n-\Ex_{n,h,\omega}[L_n]}{\sqrt{v(h)n}}\le u\bigg]-\mathcal{N}(u)\Bigg|
  +\sqrt{\frac{32}{\pi}}\int_\lambda^{+\infty}\ee^{-\frac{1}{2}u^2}\dd u\,.
\end{align}
To prove this, we appeal to an integration by parts to write for
$z,\zeta\in\Rl$
\begin{align}
  \nonumber
 \mathrm{i}z\int_{-\lambda}^\lambda\ee^{\mathrm{i}zu}\mathds{1}_{\{\zeta\le u\}}\dd u&= \ee^{\mathrm{i}z\lambda}-\ee^{\mathrm{i}z\zeta}
  +\big(\ee^{\mathrm{i}z\zeta}-\ee^{-\mathrm{i}z\lambda}\big)\mathds{1}_{\{\zeta\le-\lambda\}}
  +\big(\ee^{\mathrm{i}z\zeta}-\ee^{\mathrm{i}z\lambda}\big)\mathds{1}_{\{\zeta>\lambda\}}
\end{align}
and
\begin{align}
  \nonumber
  \mathrm{i}z\int_{-\lambda}^\lambda\ee^{\mathrm{i}zu}\mathcal{N}(u)\dd u&=\ee^{\mathrm{i}z\lambda}\mathcal{N}(\lambda)-\ee^{-\mathrm{i}z\lambda}\mathcal{N}(-\lambda)
  -\frac{1}{\sqrt{2\pi}}\int_{-\lambda}^\lambda\ee^{\mathrm{i}zu-\frac{1}{2}u^2}\dd u\\
\nonumber
&=\ee^{\mathrm{i}z\lambda}-\ee^{-\frac{1}{2}z^2}+\sqrt{\frac{2}{\pi}}\int_\lambda^{+\infty}\big[\cos(zu)-\cos(z\lambda)]\ee^{-\frac{1}{2}u^2}\dd u\,,
\end{align}
which after subtraction one to another yield
\begin{align}
  \nonumber
  \ee^{\mathrm{i}z\zeta}-\ee^{-\frac{1}{2}z^2}&=\mathrm{i}z\int_{-\lambda}^\lambda\ee^{\mathrm{i}zu}\big[\mathcal{N}(u)-\mathds{1}_{\{\zeta\le u\}}\big]\dd u
  +\sqrt{\frac{2}{\pi}}\int_\lambda^{+\infty}\big[\cos(z\lambda)-\cos(zu)\big]\ee^{-\frac{1}{2}u^2}\dd u\\
  \nonumber
  &\quad+\big(\ee^{\mathrm{i}z\zeta}-\ee^{-\mathrm{i}z\lambda}\big)\mathds{1}_{\{\zeta\le-\lambda\}}+\big(\ee^{\mathrm{i}z\zeta}-\ee^{\mathrm{i}z\lambda}\big)\mathds{1}_{\{\zeta>\lambda\}}\,.
\end{align}
Choosing $\zeta=(L_n-\Ex_{n,h,\omega}[L_n])/\sqrt{v(h)n}$ and taking
expectation, we deduce that
\begin{align}
  \nonumber
  &\Bigg|\Ex_{n,h,\omega}\bigg[\ee^{\mathrm{i}z\frac{L_n-\Ex_{n,h,\omega}[L_n]}{\sqrt{v(h)n}}}\bigg]-\ee^{-\frac{1}{2}z^2}\Bigg|\\
  \nonumber
  &\qquad\le 2|z|\lambda\,\sup_{u\in\Rl} \Bigg|\prob_{n,h,\omega}\bigg[\frac{L_n-\Ex_{n,h,\omega}[L_n]}{\sqrt{v(h)n}}\le u\bigg]-\mathcal{N}(u)\Bigg|
  +\sqrt{\frac{8}{\pi}}\int_\lambda^{+\infty}\ee^{-\frac{1}{2}u^2}\dd u\\
  \nonumber
  &\qquad\quad +2\prob_{n,h,\omega}\bigg[\frac{L_n-\Ex_{n,h,\omega}[L_n]}{\sqrt{v(h)n}}\le -\lambda\bigg]
  +2\prob_{n,h,\omega}\bigg[\frac{L_n-\Ex_{n,h,\omega}[L_n]}{\sqrt{v(h)n}}>\lambda\bigg]\\
  \nonumber
  &\qquad\le  (4+2|z|\lambda)\sup_{u\in\Rl} \Bigg|\prob_{n,h,\omega}\bigg[\frac{L_n-\Ex_{n,h,\omega}[L_n]}{\sqrt{v(h)n}}\le u\bigg]-\mathcal{N}(u)\Bigg|\\
  \nonumber
  &\qquad\quad +\sqrt{\frac{32}{\pi}}\int_\lambda^{+\infty}\ee^{-\frac{1}{2}u^2}\dd u\,.
  \qedhere
 \end{align}
\end{proof}

\medskip

\subsection{Almost microscopic contact gaps}
\label{sec:gaps}

Typical contact gaps in the localized phase are microscopic and the largest ones are almost
microscopic and well characterized (see \cite[Proposition
  1.8]{cf:companionpaper}). The following result allows us to measure
how rapidly the probability of large contact gaps decay.

\medskip

\begin{lemma}
  \label{lem:max_uncond}  
  The following property holds for $\pae$: for every closed set
  $H\subset(h_c,+\infty)$ and $\gamma>0$ there exists a constant $c>0$
  (independent of $\omega$) such that
\begin{equation*}
  \adjustlimits\limsup_{n\uparrow\infty}\sup_{h\in H}\,n^\gamma\prob_{n,h,\omega}\big[M_n>c\log n\big]<+\infty\,.
\end{equation*}
\end{lemma}

\medskip

\begin{proof}[Proof of Lemma \ref{lem:max_uncond}]
For $s\in\N$ define $h_s:=-s$ if $h_c=-\infty$ and $h_s:=h_c+1/s$ if
$h_c>-\infty$ and note that $h_s>h_c$ and
$\lim_{s\uparrow\infty}h_s=h_c$.  Then, consider the random variables
\begin{equation*}
  \Lambda_s:=\sum_{j\in\N}\ee^{\frac{1}{2}\mu(h_s)j}\,\prob_{j,h_s,\cdot}[T_1=j]\,,
\end{equation*}
$\mu(h)$ being the alternative free energy defined by
\eqref{eq:mu-def}. The variables $\Lambda_s$ possess finite
expectation because $\mu(h)>0$ for $h>h_c$, as recalled in Section
\ref{sec:richiami}, and Birkhoff's ergodic theorem ensures us that
there exists $\Omega_o\in\mathcal{F}$ with $\probd[\Omega_o]=1$ such
that for every $\omega\in\Omega_o$ and $s\in\N$
\begin{equation}
  \lim_{n\uparrow\infty}\frac{1}{n}\sum_{i=0}^{n-1}\Lambda_s(\vartheta^i\omega)=\Exd[\Lambda_s]\,.
  \label{lem:max_uncond_0}
\end{equation}

Fix $\omega\in\Omega_o$, $H\subset(h_c,+\infty)$ closed, and
$\gamma>0$. Let $s$ be so large that $[h_s,+\infty)\supset H$ and put
  $c:=(2+2\gamma)/\mu(h_s)$. We show that 
\begin{equation}
  \adjustlimits\limsup_{n\uparrow\infty}\sup_{h\in H}\,n^\gamma\prob_{n,h,\omega}\big[M_n>c\log n\big]\le \Exd[\Lambda_s]<+\infty\,.
   \label{lem:max_uncond_1}
\end{equation}
Setting $m_n:=\lfloor c\log n\rfloor$ for brevity, for every $n\in\N$
and $h\in H$ we can write
\begin{align}
  \nonumber
  \prob_{n,h,\omega}\big[M_n>c\log n\big]&=\prob_{n,h,\omega}\big[M_n>m_n\big]\\
  \nonumber
  &\le \sum_{i=0}^{n-m_n-1}\Ex_{n,h,\omega}\bigg[X_i\prod_{k=i+1}^{i+m_n}(1-X_k)\bigg]
\end{align}
since the condition $M_n>m_n$ implies that there is a contact site
$i\in\{0,\ldots,n-m_n-1\}$ followed by an excursion of size at least
$m_n$. Exploiting the identity
\begin{equation*}
  X_i\prod_{k=i+1}^{i+m_n}(1-X_k)=\sum_{j=i+m_n+1}^nX_i\prod_{k=i+1}^{j-1}(1-X_k)X_j+\prod_{k=i+1}^n(1-X_k)
\end{equation*}
and the fact that $X_n=1$ almost surely with respect to the polymer
measure $\prob_{n,h,\omega}$, we can even state that
\begin{equation*}
  \prob_{n,h,\omega}\big[M_n>c\log n\big]\le\sum_{i=0}^{n-m_n-1}\sum_{j=i+m_n+1}^n\Ex_{n,h,\omega}\bigg[X_i\prod_{k=i+1}^{j-1}(1-X_k)X_j\bigg]\,.
\end{equation*}
At this point, an application of \eqref{eq:ci-fact1} first and a
Chernoff--type bound later give
\begin{align}
  \nonumber
  \prob_{n,h,\omega}\big[M_n>c\log n\big]&\le \sum_{i=0}^{n-m_n-1}\sum_{j=i+m_n+1}^n\Ex_{j-i,h,\vartheta^i\omega}\bigg[\prod_{k=1}^{j-i-1}(1-X_k)\bigg]\\
  \nonumber
  &\le \sum_{i=0}^{n-1}\sum_{j=m_n+1}^\infty\prob_{j,h,\vartheta^i\omega}[T_1=j]\\
  \nonumber
  &\le\sum_{i=0}^{n-1}\sum_{j\in\N}\ee^{\frac{1}{2}\mu(h_s) j-\frac{1}{2}\mu(h_s)(m_n+1)}\,\prob_{j,h,\vartheta^i\omega}[T_1=j]\,.
\end{align}
Since
$\prob_{j,h,\vartheta^i\omega}[T_1=j]=p(j)/\Ex[\ee^{\sum_{a=1}^{j-1}(h+\omega_{i+a})X_a}X_j]$
is decreasing with respect to $h$, we finally deduce that
\begin{align}
  \nonumber
  \sup_{h\in H}\prob_{n,h,\omega}\big[M_n>c\log n\big]
  &\le  \ee^{-\frac{1}{2}\mu(h_s)(m_n+1)}\sum_{i=0}^{n-1}\sum_{j\in\N}\ee^{\frac{1}{2}\mu(h_s)j}\,\prob_{j,h_s,\vartheta^i\omega}[T_1=j]\\
  \nonumber
  &\le  \frac{1}{n^{\gamma+1}}\sum_{i=0}^{n-1}\Lambda_s(\vartheta^i\omega)\,.
\end{align}
This bound implies \eqref{lem:max_uncond_1} thanks to
\eqref{lem:max_uncond_0}.
\end{proof}

\medskip

We are finally able to prove Theorem \ref{th:main} by means of a
change of measure.

\medskip

\begin{proof}[Proof of Theorem \ref{th:main}]
To begin with, we note that the mean contact number
$\Ex_{n,h,\omega}[L_n]$ with given $n\in\N$ and $\omega\in\Omega$
increases continuously with respect to $h$ and the limits
$\lim_{h\downarrow-\infty}\Ex_{n,h,\omega}[L_n]=1$ and
$\lim_{h\uparrow+\infty}\Ex_{n,h,\omega}[L_n]=n$ hold true. Thus, for
each $l\in\{2,\ldots,n-1\}$ there exists a unique number
$h_{n,l,\omega}$ such that
$\Ex_{n,h_{n,l,\omega},\omega}[L_n]=l$. This number allows us a
profitable change of measure, which in turn originates the following
bound that we use to prove the theorem: for every $n\in\N$,
$h,c\in\Rl$, $\omega\in\Omega$, and $l\in\{2,\ldots,n-1\}$
\begin{align}
  \nonumber
  \prob_{n,h,\omega}\big[M_n>c\log n\,\big| L_n = l\big]&= \frac{\prob_{n,h_{n,l,\omega},\omega}[M_n>c\log n\,,L_n=l]}{\prob_{n,h_{n,l,\omega},\omega}[L_n=l]}\\
  &\le \frac{\prob_{n,h_{n,l,\omega},\omega}[M_n>c\log n]}{\prob_{n,h_{n,l,\omega},\omega}[L_n=l]}\,.
\label{th:main_1}
\end{align}

\smallskip

Theorems \ref{th:Cinfty} and \ref{th:LCLT} and Lemma
\ref{lem:max_uncond} assure us that there exists
$\Omega_o\in\mathcal{F}$ with $\probd[\Omega_o]=1$ such that the
following properties hold for all $\omega\in\Omega_o$,
$H\subset(h_c,+\infty)$ compact, and some constant $c>0$ possibly
depending on $H$ but not on $\omega$:
\vspace{0.1cm}
\begin{enumerate}[itemsep=0.5em]
\item $\displaystyle{\adjustlimits\lim_{n\uparrow\infty}\sup_{h\in H}\bigg|\Ex_{n,h,\omega}\bigg[\frac{L_n}{n}\bigg]-\rho(h)\bigg|=0\,;}$
  \item $\displaystyle{\adjustlimits\limsmallspace_{n\uparrow\infty}
    \sup_{h\in H}\,\sup_{l\in\N_0}\Bigg|\sqrt{2\pi v(h) n}\,\prob_{n,h,\omega}[L_n=l]-\exp \bigg\{\!-\frac{(l-\Ex_{n,h,\omega}[L_n])^2}{2 v(h) n}\bigg\}\Bigg|=0\,;}$
    \vspace{0.25cm}
    \item  $\displaystyle{ \adjustlimits\limsup_{n\uparrow\infty}\sup_{h\in H}\,n\,\prob_{n,h,\omega}\big[M_n>c\log n\big]<+\infty\,.}$
\end{enumerate}
\vspace{0.1cm} 
Fix $\omega\in\Omega_o$ and suppose without restriction
that the compact set $R$ in the statement of Theorem \ref{th:main} is
of the form $[\delta,1-\delta]$ with some $\delta\in(0,1/2)$. In this
way, $nR \cap \N$ becomes the set of all integers $l$ such that
$\delta n\le l\le (1-\delta)n$.  Since
$\int_\Omega\omega_0^2\,\probd[\dd\omega]>0$, according to Theorem
\ref{th:smoothing} we have $\rho_c=0$ and, whatever $\delta$ is, there
are values of the free parameter $h$ in the localized phase where the
contact density attains the values $\delta/2$ and $1-\delta/2$. Such
values are $\imath_\rho(\delta/2)>h_c$ and
$\imath_\rho(1-\delta/2)>\imath_\rho(\delta/2)$, $\imath_\rho$ being the inverse
function of the contact density introduced in Section \ref{sec:LDP}.
Property $(1)$ tells us that
$\lim_{n\uparrow\infty}\Ex_{n,\imath_\rho(\delta/2),\omega}[L_n/n]=\delta/2$
and
$\lim_{n\uparrow\infty}\Ex_{n,\imath_\rho(1-\delta/2),\omega}[L_n/n]=1-\delta/2$,
so that $\Ex_{n,\imath_\rho(\delta/2),\omega}[L_n]\le\delta n$ and
$\Ex_{n,\imath_\rho(1-\delta/2),\omega}[L_n]\ge (1-\delta)n$ for all
integers $n\ge n_o$ with some number $n_o$. Setting
$H:=[\imath_\rho(\delta/2),\imath_\rho(1-\delta/2)]$, it therefore follows that
the above number $h_{n,l,\omega}$ belongs to $H$ for all integers
$n\ge n_o$ and $l\in  nR \cap \N$. Let $c>0$ be a constant
(independent of $\omega$) such that Property $(3)$ holds. For all
integers $n\ge n_o$ and $l\in  nR \cap \N$ we have
\begin{equation}
  \prob_{n,h_{n,l,\omega},\omega}\big[M_n>c\log n\big]\le \sup_{h\in H}\prob_{n,h,\omega}\big[M_n>c\log n\big]
\label{th:main_100}
\end{equation}
and, recalling that
$\Ex_{n,h_{n,l,\omega},\omega}[L_n]=l$ by construction, 
\begin{align}
  \nonumber
  &\sqrt{2\pi\sup_{h\in H}v(h)n}\,\prob_{n,h_{n,l,\omega},\omega}[L_n=l]\\
  \nonumber
  &\qquad\ge\sqrt{2\pi v(h_{n,l,\omega})n}\,\prob_{n,h_{n,l,\omega},\omega}[L_n=l]\\
  &\qquad\ge 1-\sup_{h\in H}\,\sup_{l\in\N_0}\Bigg|\sqrt{2\pi v(h) n}\,\prob_{n,h,\omega}[L_n=l]-\exp \bigg\{\!-\frac{(l-\Ex_{n,h,\omega}[L_n])^2}{2 v(h) n}\bigg\}\Bigg|\,.
\label{th:main_101}
\end{align}
Combining bounds \eqref{th:main_100} and \eqref{th:main_101} together
and invoking Properties $(2)$ and $(3)$ we get
\begin{equation*}
  \adjustlimits\limsmallspace_{n\uparrow\infty}\sup_{l\in  nR \cap \N}\,\frac{\prob_{n,h_{n,l,\omega},\omega}[M_n>c\log n]}{\prob_{n,h_{n,l,\omega},\omega}[L_n=l]}=0\,.
\end{equation*}
This proves the theorem thanks to bound \eqref{th:main_1}.
\end{proof}

\subsection{Mesoscopic contact averages}
\label{sec:meso}

Theorem~\ref{th:smallscales} follows from the next finding,
Theorem~\ref{th:small_scales}, exactly in the same way as
Theorem~\ref{th:main} follows from Lemma~\ref{lem:max_uncond}. But
Theorem~\ref{th:small_scales} is a result of interest in its own right
as it shows that the macroscopic behavior of the (unconditioned)
pinning model is already detectable at mesoscopic scales that are just
required to grow faster than $\log n$.

\medskip

\begin{theorem}
  \label{th:small_scales}
  The following property holds for $\pae$: for every compact set
  $H\subset(h_c,+\infty)$, all numbers $\gamma,\epsilon>0$, and any
  sequence $\{\zeta_n\}_{n\in\N}$ of positive numbers such that
  $\lim_{n\uparrow\infty}\zeta_n=+\infty$
  \begin{equation*}
    \adjustlimits\lim_{n\uparrow\infty}\sup_{h\in H}n^\gamma\,
    \prob_{n,h,\omega}\Bigg[\max_{\substack{0\le i<j\le n\\j-i\ge \zeta_n\log n}}\bigg|\frac{1}{j-i}\sum_{a=i+1}^jX_a-\rho(h)\bigg|>\epsilon\Bigg]<+\infty\,.
\end{equation*}
  \end{theorem}

\medskip

We present a proof that relies on the following preliminary lemma.

\medskip


\begin{lemma}
  \label{lemma:Ex_large_deviation}
  For every compact set $H\subset(h_c,+\infty)$ and $\epsilon>0$ there
  exist constants $\gamma>0$ and $G>0$ such that for all $n\in\N$
  \begin{equation*}
   \sup_{h\in H}\prob_{n,h,\cdot}\big[|L_n-\rho(h)n|>\epsilon n\big]\le \Gamma_n
  \end{equation*}
  with a non-negative measurable function $\Gamma_n$ on $\Omega$ that
  satisfies $\Exd[\Gamma_n]\le G\ee^{-\gamma n}$.
\end{lemma}

\medskip

\begin{proof}[Proof Lemma \ref{lemma:Ex_large_deviation}]
  Fix $H\subset(h_c,+\infty)$ compact and $\epsilon>0$.  We shall show
  that for every $h\in H$ there exist constants $\gamma_h>0$ and
  $G_h>0$ such that for all $n\in\N$
  \begin{equation}
    \Exd\Big[\prob_{n,h,\cdot}\big[|L_n-\rho(h)n|>\epsilon n/2\big]\Big]\le G_h\ee^{-2\gamma_h n}\,.
    \label{lemma:Ex_large_deviation_0}
  \end{equation}
  This proves the lemma as follows.  Since $\rho:=\partial_hf$ is
  uniformly continuous on the compact set $H$ by Theorem
  \ref{th:Cinfty}, there exists $\delta_o>0$ such that
  $|\rho(h)-\rho(h')|\le\epsilon/2$ if $h,h'\in H$ satisfy
  $|h-h'|<2\delta_o$.  Put $\delta_h:=\min\{\gamma_h,\delta_o\}$ and
  note that we can find finitely many points $h_1,\ldots,h_r$ in $H$
  such that
  $H\subset\cup_{s=1}^r(h_s-\delta_{h_s},h_s+\delta_{h_s})$. For all
  $n\in\N$, $h\in(h_s-\delta_{h_s},h_s+\delta_{h_s})\cap H$, and
  $\omega\in\Omega$ we have
\begin{align}
  \nonumber
  \prob_{n,h,\omega}\big[|L_n-\rho(h)n|>\epsilon n\big]&\le\prob_{n,h,\omega}\big[|L_n-\rho(h_s)n|>\epsilon n/2\big]\\
  \nonumber
  &=\frac{\Ex[\mathds{1}_{\{|L_n-\rho(h_s)n|>\epsilon n/2\}}\ee^{\sum_{a=1}^n(h+\omega_a)X_a}X_n]}{\Ex[\ee^{\sum_{a=1}^n(h+\omega_a)X_a}X_n]}\\
  \nonumber
  &\le\ee^{\delta_{h_s}n}\prob_{n,h_s,\omega}\big[|L_n-\rho(h_s)n|>\epsilon n/2\big]\,,
\end{align}
so
\begin{equation*}
  \sup_{h\in H}\prob_{n,h,\omega}\big[|L_n-\rho(h)n|>\epsilon n\big]\le\sum_{s=1}^r\ee^{\delta_{h_s}n}\prob_{n,h_s,\omega}\big[|L_n-\rho(h_s)n|>\epsilon n/2\big]=:\Gamma_n(\omega)\,.
\end{equation*}
The function $\Gamma_n$ that maps $\omega\in\Omega$ to
$\Gamma_n(\omega)$ is non-negative and measurable and, due to
\eqref{lemma:Ex_large_deviation_0}, satisfies $\Exd[\Gamma_n]\le
G\ee^{-\gamma n}$ with
$\gamma:=\min\{\gamma_{h_1},\ldots,\gamma_{h_r}\}$ and
$G:=G_{h_1}+\cdots+G_{h_r}$.


\smallskip

\noindent\textit{The bound \eqref{lemma:Ex_large_deviation_0}}. Pick
$h\in H$ and let $\zeta>0$ be a number so small that
$f(h+\zeta)-f(h)-\rho(h)\zeta\le\epsilon\zeta/8$ and
$f(h-\zeta)-f(h)+\rho(h)\zeta\le\epsilon\zeta/8$, which exists because
$\rho(h):=\partial_h f(h)$. Put $c:=\epsilon\zeta/32$ for brevity. If
$\log Z_{n,h}(\omega)\ge \Exd[\log Z_{n,h}]-c n$ and $\log
Z_{n,h+\zeta}(\omega)\le\Exd[\log Z_{n,h+\zeta}]+c n$, then a
Chernoff--type bound gives
\begin{align}
    \nonumber
    \prob_{n,h,\omega}\big[L_n>\rho(h)n+\epsilon n/2\big]&\le \Ex_{n,h,\omega}\big[\ee^{\zeta L_n}\big]\ee^{-\rho(h)\zeta n+\epsilon\zeta n/2}\\
    \nonumber
    &=\frac{Z_{n,h+\zeta}(\omega)}{Z_{n,h}(\omega)}\ee^{-\rho(h)\zeta n+\epsilon\zeta n/2}\\
    \nonumber
  &\le\ee^{\Exd[\log Z_{n,h+\zeta}]-\Exd[\log Z_{n,h}]-\rho(h)\zeta n-\epsilon\zeta n/2+2c n}.
  \end{align}
This allows us to state that for all $n\in\N$ and $\omega\in\Omega$
\begin{align}
    \nonumber
    \prob_{n,h,\omega}\big[L_n>\rho(h)n+\epsilon n/2\big]&\le\mathds{1}_{\{\log Z_{n,h}(\omega)<\Exd[\log Z_{n,h}]-c n\}}+\mathds{1}_{\{\log Z_{n,h+\zeta}(\omega)>\Exd[\log Z_{n,h+\zeta}]+c n\}}\\
    \nonumber
    &\quad+\ee^{\Exd[\log Z_{n,h+\zeta}]-\Exd[\log Z_{n,h}]-\rho(h)\zeta n-\epsilon\zeta n/2+2c n}\,.
  \end{align}
Similarly, we can state that for all $n\in\N$ and $\omega\in\Omega$
   \begin{align}
    \nonumber
    \prob_{n,h,\omega}\big[L_n<\rho(h)n-\epsilon n/2\big]&\le\mathds{1}_{\{\log Z_{n,h}(\omega)<\Exd[\log Z_{n,h}]-\delta n\}}+\mathds{1}_{\{\log Z_{n,h-\zeta}(\omega)>\Exd[\log Z_{n,h-\zeta}]+\delta n\}}\\
    \nonumber
    &\quad+\ee^{\Exd[\log Z_{n,h-\zeta}]-\Exd[\log Z_{n,h}]+\rho(h)\zeta n-\epsilon\zeta n/2 +2c n}\,.
   \end{align}
Integrating with respect to $\probd[\dd\omega]$, appealing to the
concentration inequality given by Theorem \ref{th:concentration}, and
using that $\Exd[\log Z_{n,h+\zeta}]\le nf(h+\zeta)+cn$, $\Exd[\log
  Z_{n,h-\zeta}]\le nf(h-\zeta)+cn$, and $\Exd[\log Z_{n,h}]\ge
nf(h)-cn$ for all sufficiently large $n$, we get for those $n$
 \begin{equation}
    \Exd\Big[\prob_{n,h,\cdot}\big[L_n>\rho(h)n+\epsilon n/2\big]\Big]\le 4\ee^{-\frac{\kappa c^2}{1+c}n}+\ee^{[f(h+\zeta)-f(h)-\rho(h)\zeta-\epsilon\zeta/2 +4c] n}
\label{lemma:Ex_large_deviation_1}
 \end{equation}
and
 \begin{equation}
    \Exd\Big[\prob_{n,h,\cdot}\big[L_n<\rho(h)n-\epsilon n/2\big]\Big]\le 4\ee^{-\frac{\kappa c^2}{1+c}n}+\ee^{[f(h-\zeta)-f(h)+\rho(h)\zeta-\epsilon\zeta/2 +4c] n}\,.
\label{lemma:Ex_large_deviation_2}
 \end{equation}
 By construction we have $f(h+\zeta)-f(h)-\rho(h)\zeta-\epsilon\zeta/2
 +4c\le -\epsilon\zeta/4$ and
 $f(h-\zeta)-f(h)+\rho(h)\zeta-\epsilon\zeta/2 +4c\le
 -\epsilon\zeta/4$, so \eqref{lemma:Ex_large_deviation_1} and
 \eqref{lemma:Ex_large_deviation_2} show that
 \eqref{lemma:Ex_large_deviation_0} is demonstrated with
 $\gamma_h:=\min\{\kappa c^2/(1+c),\epsilon\zeta/4\}$ and some
 constant $G_h>0$.
\end{proof}

\medskip

\begin{proof}[Proof of Theorem \ref{th:small_scales}]
For $s\in\N$ define $h_s:=-s$ if $h_c=-\infty$ and $h_s:=h_c+1/s$ if
$h_c>-\infty$ and note that $h_s>h_c$,
$\lim_{s\uparrow\infty}h_s=h_c$, and
$\lim_{s\uparrow\infty}(h_s+2s)=+\infty$. According to Lemma
\ref{lemma:Ex_large_deviation} there exist constants $\gamma_s>0$ and
$G_s>0$ such that for all $j\in\N$
 \begin{equation}
 \sup_{h\in [h_s,h_s+2s]}\prob_{j,h,\cdot}\big[|L_j-\rho(h)j|>j/s\big]\le \Gamma_{s,j}
\label{eq:small_scales_1}
 \end{equation}
with a non-negative measurable function $\Gamma_{s,j}$ on $\Omega$
that satisfies $\Exd[\Gamma_{s,j}]\le G_s\ee^{-\gamma_s j}$.  The
random variable $\Lambda_s:=\sum_{j\in\N}\ee^{\gamma_s
  j/2}\,\Gamma_{s,j}$ possesses finite expectation, so Birkhoff's
ergodic theorem assures us that there exists a set
$\Omega_o\in\mathcal{F}$ with $\probd[\Omega_o]=1$ such that for all
$\omega\in\Omega_o$ and $s\in\N$
\begin{equation}
  \sup_{n\in\N}\,\frac{1}{n}\sum_{i=0}^{n-1}\Lambda_s(\vartheta^i\omega)<+\infty\,.
\label{eq:small_scales_2}
\end{equation}
According to Lemma \ref{lem:max_uncond} we can also suppose that
$\Omega_o$ has the following property: for every $\omega\in\Omega_o$,
compact set $H\subset(h_c,+\infty)$, and $\gamma>0$ there exists a
constant $c>0$ (independent of $\omega$) such that
\begin{equation}
  \adjustlimits\limsup_{n\uparrow\infty}\sup_{h\in H}\,n^\gamma\prob_{n,h,\omega}\big[M_n>c\log n\big]<+\infty\,.
\label{eq:small_scales_3}
\end{equation}

Fix $\omega\in\Omega$, $H\subset(h_c,+\infty)$ compact, and
$\gamma,\epsilon>0$. Also fix an arbitrary sequence
$\{\zeta_n\}_{n\in\N}$ of positive numbers such that
$\lim_{n\uparrow\infty}\zeta_n=+\infty$.  Let $c>0$ be as in
\eqref{eq:small_scales_3} and put $m_n:=\lfloor c\log n\rfloor$ for
brevity. Let $n_o$ be a number so large that $4(1+\epsilon)m_n+2\le
\epsilon\zeta_n\log n$ for all integers $n>n_o$.  For every integer
$n>n_o$ and $h\in H$ we have
\begin{align}
  \nonumber
  &\prob_{n,h,\omega}\Bigg[\max_{\substack{0\le i<j\le n\\j-i\ge \zeta_n\log n}}\bigg|\frac{1}{j-i}\sum_{a=i+1}^jX_a-\rho(h)\bigg|>\epsilon\Bigg]\\
  \nonumber
  &\qquad\le\prob_{n,h,\omega}[M_n>c\log n]\\
  &\qquad\quad+\sum_{i=0}^{n-1}\sum_{j=i+1}^n\mathds{1}_{\{j-i\ge\zeta_n\log n\}}\prob_{n,h,\omega}\bigg[\bigg|\frac{1}{j-i}\sum_{a=i+1}^jX_a-\rho(h)\bigg|>\epsilon,\,M_n\le m_n\bigg]\,.
\label{eq:small_scales_4}
\end{align}
Let us elaborate the terms in the double sum.  Under the conditions
$X_0=X_n=1$ and $M_n\le m_n$ we have the identities
\begin{equation*}
  \sum_{i'=\max\{0,i-m_n+1\}}^iX_{i'}\prod_{k=i'+1}^i(1-X_k)=\sum_{j'=j}^{\min\{j+m_n-1,n\}}\prod_{k=j}^{j'-1}(1-X_k)X_{j'}=1\,,
\end{equation*}
so for $0\le i<j\le n$ with $j-i\ge\zeta_n\log n$ we can write down
\begin{align}
  \nonumber
  &\prob_{n,h,\omega}\bigg[\bigg|\frac{1}{j-i}\sum_{a=i+1}^jX_a-\rho(h)\bigg|>\epsilon,\,M_n\le m_n\bigg]\\
  \nonumber
  &\qquad\le\sum_{i'=\max\{0,i-m_n+1\}}^i\sum_{j'=j}^{\min\{j+m_n-1,n\}}\Ex_{n,h,\omega}\Bigg[X_{i'}\prod_{k=i'+1}^i(1-X_k)\\
    \nonumber
    &\qquad\qquad\qquad\qquad\qquad\qquad\qquad\qquad\times\mathds{1}_{\big\{\big|\frac{1}{j-i}\sum_{a=i+1}^jX_a-\rho(h)\big|>\epsilon\big\}}\prod_{k=j}^{j'-1}(1-X_k)X_{j'}\Bigg]\\
  \nonumber
  &\qquad\le\sum_{i'=\max\{0,i-m_n+1\}}^i\sum_{j'=j}^{\min\{j+m_n-1,n\}}\Ex_{n,h,\omega}\bigg[X_{i'}
    \mathds{1}_{\big\{\big|\frac{1}{j-i}\left(\sum_{a=i'+1}^{j'-1}X_a+\mathds{1}_{\{j'=j\}}\right)-\rho(h)\big|>\epsilon\big\}}X_{j'}\bigg]\,.
\end{align}
Now we note that
$|\sum_{a=i'+1}^{j'-1}X_a+\mathds{1}_{\{j'=j\}}-\rho(h)(j-i)|>\epsilon(j-i)$
with $i-m_n<i'\le i$ and $j<j'\le j+m_n$ implies
$|\sum_{a=i'+1}^{j'}X_a-\rho(h)(j'-i')|>\epsilon(j'-i')-2(1+\epsilon)
m_n-1$, where we have used that $0\le \rho(h)\le 1$. But for the
integers $n>n_o$ the condition $j-i\ge\zeta_n\log n$ entails
$4(1+\epsilon)m_n+2\le\epsilon\zeta_n\log n\le
\epsilon(j-i)\le\epsilon(j'-i')$. In conclusion, we can state that for
$n>n_o$ and $0\le i<j\le n$ with $j-i\ge\zeta_n\log n$
\begin{align}
  \nonumber
  &\prob_{n,h,\omega}\bigg[\bigg|\frac{1}{j-i}\sum_{a=i+1}^jX_a-\rho(h)\bigg|>\epsilon,\,M_n\le m_n\bigg]\\
  \nonumber
  &\qquad\le\sum_{i'=\max\{0,i-m_n+1\}}^i\sum_{j'=j}^{\min\{j+m_n-1,n\}}\Ex_{n,h,\omega}\bigg[X_{i'}
    \mathds{1}_{\big\{\big|\sum_{a=i'+1}^{j'}X_a-\rho(h)(j'-i')\big|>\frac{\epsilon}{2}\epsilon(j'-i')\big\}}X_{j'}\bigg]\,.
\end{align}
At this point, two applications of \eqref{eq:ci-fact2} give
\begin{align}
  \nonumber
  &\prob_{n,h,\omega}\bigg[\bigg|\frac{1}{j-i}\sum_{a=i+1}^jX_a-\rho(h)\bigg|>\epsilon,\,M_n\le m_n\bigg]\\
  \nonumber
   &\qquad\le\sum_{i'=\max\{0,i-m_n+1\}}^i\sum_{j'=j+1}^{\min\{j+m_n,n\}}\prob_{j'-i',h,\vartheta^{i'}\omega}\Big[\big|L_{j'-i'}-\rho(h)(j'-i')\big|>\epsilon(j'-i')/2\Big]\,.
\end{align}
Combining this bound with \eqref{eq:small_scales_4} we finally deduce
that for all integers
$n>n_o$ and $h\in H$
\begin{align}
  \nonumber
  &\prob_{n,h,\omega}\Bigg[\max_{\substack{0\le i<j\le n\\j-i\ge \zeta_n\log n}}\bigg|\frac{1}{j-i}\sum_{a=i+1}^jX_a-\rho(h)\bigg|>\epsilon\Bigg]\\
  \nonumber
  &\qquad\le\prob_{n,h,\omega}[M_n>c\log n]\\
  &\qquad\quad+m_n^2\sum_{i=0}^{n-1}\sum_{j\in\N}\mathds{1}_{\{j\ge\zeta_n\log n\}}\prob_{j,h,\vartheta^i\omega}\Big[\big|L_j-\rho(h)j\big|>\epsilon j/2\Big]\,.
\label{eq:small_scales_5}
\end{align}

For the last step we pick $s\in\N$ so large that
$H\subseteq[h_s,h_s+2s]$ and $\epsilon/2\ge 1/s$. Bounds
\eqref{eq:small_scales_5} and \eqref{eq:small_scales_1} show that for
all integers $n>n_o$
\begin{align}
  \nonumber
  &\sup_{h\in H}\prob_{n,h,\omega}\Bigg[\max_{\substack{0\le i<j\le n\\j-i\ge \zeta_n\log n}}\bigg|\frac{1}{j-i}\sum_{a=i+1}^jX_a-\rho(h)\bigg|>\epsilon\Bigg]\\
  \nonumber
  &\qquad\le\sup_{h\in H}\prob_{n,h,\omega}[M_n>c\log n]+m_n^2\sum_{i=0}^{n-1}\sum_{j\in\N}\mathds{1}_{\{j\ge\zeta_n\log n\}}\Gamma_{s,j}(\vartheta^i\omega)\,,
\end{align}
and a Chernoff--type bound, together with the definition
$\Lambda_s:=\sum_{j\in\N}\ee^{\gamma_s j/2}\,\Gamma_{s,j}$, yields
\begin{align}
  \nonumber
  &\sup_{h\in H}\prob_{n,h,\omega}\Bigg[\max_{\substack{0\le i<j\le n\\j-i\ge \zeta_n\log n}}\bigg|\frac{1}{j-i}\sum_{a=i+1}^jX_a-\rho(h)\bigg|>\epsilon\Bigg]\\
  \nonumber
  &\qquad\le\sup_{h\in H}\prob_{n,h,\omega}[M_n>c\log n]+m_n^2\ee^{-\frac{1}{2}\gamma_s\zeta_n\log n}\sum_{i=0}^{n-1}\Lambda_s(\vartheta^i\omega)\,.
\end{align}
From here we find
 \begin{equation*}
    \adjustlimits\lim_{n\uparrow\infty}\sup_{h\in H}n^\gamma\,
    \prob_{n,h,\omega}\Bigg[\max_{\substack{0\le i<j\le n\\j-i\ge \zeta_n\log n}}\bigg|\frac{1}{j-i}\sum_{a=i+1}^jX_a-\rho(h)\bigg|>\epsilon\Bigg]<+\infty
 \end{equation*}
 thanks to \eqref{eq:small_scales_2} and \eqref{eq:small_scales_3} and
 thanks to the fact that $\lim_{n\uparrow\infty}\zeta_n=+\infty$.
\end{proof}

\section{Generalizations}
\label{sec:generalizations}

\subsection{The soft conditioning case}

The findings about sharp conditioning, i.e., Proposition
\ref{prop:no_disorder} and Theorem \ref{th:main}, allow easily to
obtain similar results for a soft conditioning that restores the role
of the parameter $h$. Here we discuss this issue for the disordered
model on the basis of Proposition \ref{th:LD}, which is used to
restrict the range of the contact number, and Theorem
\ref{th:main}. The pure model is not considered for conciseness,
 but the same arguments based on Propositions \ref{th:LD} and
\ref{prop:no_disorder} apply to it. The following proposition
exemplifies two forms of soft conditioning.

\medskip

\begin{proposition}
  \label{prop:soft}
  Suppose that $\int_\Omega\omega_0^2\,\probd[\dd\omega]>0$. The
  following property holds for $\pae$: for every $h\in\Rl$ and
  $r\in(0,1)$ there exists a constant $c>0$ (independent of $\omega$)
  such that
  \begin{equation*}
    \lim_{n\uparrow\infty} \prob_{n,h,\omega}\big[M_n>c\log n\big| L_n \ge r n\big]=0
  \end{equation*}
  and, when $h>h_c$,
  \begin{equation*}
    \lim_{n\uparrow\infty}\prob_{n,h,\omega}\big[M_n>c\log n\big| L_n \le r n\big]=0\,.
  \end{equation*}
\end{proposition}

\medskip

\begin{proof}[Proof of Proposition \ref{prop:soft}]
According to Proposition \ref{th:LD} and Theorem \ref{th:main} there
exists a set $\Omega_o\in\mathcal{F}$ with $\probd[\Omega_o]=1$ such
that the following properties hold for every $\omega\in\Omega_o$,
$h\in\Rl$, $G\subseteq\Rl$ open, $F\subseteq\Rl$ closed,
$R\subset(0,1)$ closed, and some constant $c>0$ depending only on $R$:
  \vspace{0.1cm}
  \begin{enumerate}[itemsep=0.3em]
 \item    $\displaystyle{\liminf_{n\uparrow\infty} \frac{1}{n}\log\prob_{n,h,\omega}\bigg[\frac{L_n}{n}\in G\bigg] \ge-\inf_{r\in G} I_h(r)\,;}$
\item $\displaystyle{\limsup_{n\uparrow\infty} \frac{1}{n}\log\prob_{n,h,\omega}\bigg[\frac{L_n}{n}\in F\bigg] \le-\inf_{r\in F} I_h(r)\,;}$  
\item $\displaystyle{\adjustlimits\limsmallspace_{n\uparrow\infty}\sup_{l\in nR\cap\N}\,\prob_{n,h,\omega}\big[M_n>c\log n\big| L_n = l\big]=0\,.}$
  \end{enumerate}
  \vspace{0.1cm}
  Moreover, Theorem \ref{th:smoothing} yields $\rho_c=0$, as
  $\int_\Omega\omega_0^2\,\probd[\dd\omega]>0$, and formula
  \eqref{eq:I_h_esplicita} shows that $\partial_r
  I_h(r)=\imath_\rho(r)-h$ for all $r\in(0,1)$, so $I_h$ is strictly
  increasing on $[0,1]$ if $h\le h_c$, whereas it is strictly
  decreasing on $[0,\rho(h)]$ and strictly increasing on $[\rho(h),1]$
  if $h>h_c$.

Fix $\omega\in\Omega_o$, $h\in\Rl$, and $r\in(0,1)$. The shape of the
graph of $I_h$ assures us that there exists $r_+\in(r,1)$ such that
$I_h(r_+)>I_h(r)$ and, when $h>h_c$, that there exists $r_-\in(0,r)$
such that $I_h(r_-)>I_h(r)$. Thus, Properties $(1)$ and $(2)$ give
\begin{equation}
  \lim_{n\uparrow\infty}\frac{\prob_{n,h,\omega}[L_n\ge r_+]}{\prob_{n,h,\omega}[L_n>r]}=0
  \label{eq:soft_1}
\end{equation}
and, when $h>h_c$,
\begin{equation}
\lim_{n\uparrow\infty}\frac{\prob_{n,h,\omega}[L_n\le r_-]}{\prob_{n,h,\omega}[L_n<r]}=0\,.
\label{eq:soft_2}
\end{equation}
If $h>h_c$, then let $c>0$ be the constant that Property $(3)$
associates with $R_+:=[r,r_+]\subset(0,1)$, whereas if $h\le h_c$, then let
$c>0$ be the maximum between the constants that Property $(3)$
associates with $R_+$ and $R_-:=[r_-,r]\subset(0,1)$. For all $n$ we
have
\begin{align}
  \nonumber
  &\prob_{n,h,\omega}\big[M_n>c\log n\big|L_n\ge r n\big]\\
  \nonumber
  &\qquad\le \frac{\prob_{n,h,\omega}[M_n>c\log n,\, L_n/n\in R_+]}{{\prob_{n,h,\omega}[L_n\ge r]}}
  +\frac{\prob_{n,h,\omega}\big[L_n\ge r_+n\big]}{\prob_{n,h,\omega}[L_n>r]}\\
  \nonumber
  &\qquad\le \sup_{l\in nR_+\cap\N}\,\prob_{n,h,\omega}\big[M_n>c\log n\big|L_n=l\big]+\frac{\prob_{n,h,\omega}\big[L_n\ge r_+n\big]}{\prob_{n,h,\omega}[L_n>r]}\,,
\end{align}
and, similarly when $h>h_c$,
\begin{align}
  \nonumber
  &\prob_{n,h,\omega}\big[M_n>c\log n\big|L_n\le r n\big]\\
  \nonumber
  &\qquad\le \sup_{l\in nR_-\cap\N}\,\prob_{n,h,\omega}\big[M_n>c\log n\big|L_n=l\big]+\frac{\prob_{n,h,\omega}\big[L_n\le r_-n\big]}{\prob_{n,h,\omega}[L_n<r]}\,.
\end{align}
These bounds prove that $\lim_{n\uparrow\infty}
\prob_{n,h,\omega}[M_n>c\log n| L_n \ge r n]=0$ and, when $h>h_c$,
$\lim_{n\uparrow\infty}\prob_{n,h,\omega}[M_n>c\log n| L_n \le r
  n]=0$ thanks to \eqref{eq:soft_1} and \eqref{eq:soft_2} and thanks
to Property $(3)$.
\end{proof}

\medskip

\subsection{Circular DNA model(s)}
Let $U$ be a real function on the interval $(0,1]$, and for $n\in\N$,
$h\in\Rl$, and $\omega:=\{\omega_a\}_{a\in\N_0}\in\Omega$ consider the
model
\begin{equation}
\label{eq:U-model}
  \frac{\dd\prob_{n,\omega}^U}{\dd\prob}:=\frac{1}{Z_n^U(\omega)}\ee^{nU(L_n/n)+\sum_{i=1}^{L_n}\omega_{S_i}}\mathds{1}_{\{n\in S\}}\,,
\end{equation}
where of course
$Z_n^U(\omega):=\Ex[\ee^{nU(L_n/n)+\sum_{i=1}^{L_n}\omega_{S_i}}\mathds{1}_{\{n\in S
    \}}]$. We recover the pinning model with parameter $h\in\Rl$ when
$U(r)=hr$ and \eqref{eq:U-model} may be seen as a natural
\emph{nonlinear} generalization of this model. Note that the nonlinear
dependence is on the contact density, computed over the whole length
of the polymer, and this term has therefore a nonlocal nature. The
model \eqref{eq:U-model} has been introduced in the biophysics
literature for special cases of the potential $U$ to study the
denaturation for circular DNA (see, e.g., \cite{cf:bar2,giacomin2020}
and references therein).  The fact that the two DNA strands are in a
ring form, coupled with the double helix form, makes the opening of
the two strands a complex and nonlocal phenomenon, because the two
strands are entangled: it is what we experience when we try to
disentangle two ropes. This issue is tackled in several applied
science publications and we refer to \cite{giacomin2020} for a
discussion in a more mathematical spirit.

\smallskip

The following proposition describes the maximal contact gap in the
generalized disordered model under the assumption that there exists a
closed set $R\subset(0,1)$ such that for $\pae$
\begin{equation}
  \lim_{n\uparrow\infty} \prob^U_{n,\omega}\bigg[\frac{L_n}{n}\in R\bigg]=1\,. 
\label{eq:HyU}
\end{equation}
For instance, this is the case if there exist $h>h_c$, an open set
$G\subset(0,1)$ containing $\rho(h)$, and a closed set $F\subset(0,1)$
such that
\begin{equation}
 \sup_{r\in (0,1]\setminus F}\big\{U(r)-hr\big\}\le\infp_{r\in G~~}\!\!\!\!\big\{U(r)-hr\big\}\,.
\label{eq:HyU_example}
\end{equation}
In fact, shifting $U$ by $h$ one can work with a reference measure
corresponding to a localized pinning model whose contact density is in
$G$ with full probability at large $n$, and then one can take
$R:=F$. The potentials $U$ involved in \cite{giacomin2020} satisfy
\eqref{eq:HyU_example}.

\medskip

\begin{proposition}
  \label{prop:no_bigjump_disorder_generalized}
  Suppose that $\int_\Omega\omega_0^2\,\probd[\dd\omega]>0$ and that
  \eqref{eq:HyU} holds. There exists a constant $c>0$ such that for
  $\pae$
  \begin{equation*}
    \lim_{n\uparrow\infty}\prob^U_{n,\omega}\big[M_n\le c\log n\big]=1\,.
  \end{equation*}
\end{proposition}

\medskip

\begin{proof}[Proof of Proposition \ref{prop:no_bigjump_disorder_generalized}]
Let $R\subset(0,1)$ be as in \eqref{eq:HyU}. Theorem \ref{th:main}
assures us that there exists a constant $c>0$ such that for $\pae$
 \begin{equation*}
    \adjustlimits\limsmallspace_{n\uparrow\infty} \sup_{l\in nR\cap\N}\prob_{n,0,\omega}\Big[M_n>c\log n\,\Big| L_n = l\Big]=0\,.
  \end{equation*}
  On the other hand, for all $n\in\N$ and $\omega\in\Omega$ we have
\begin{equation*}
 \prob^U_{n,\omega}\big[M_n>c\log n\big]\le \prob^U_{n,\omega}\bigg[\frac{L_n}{n}\notin R\bigg] +\prob^U_{n,\omega}\bigg[M_n>c\log n,\frac{L_n}{n}\in R\bigg] 
\end{equation*}
with 
\begin{align}
  \nonumber
  \prob^U_{n,\omega}\bigg[M_n>c\log n,\frac{L_n}{n}\in R\bigg] &=\sum_{l\in nR\cap\N}\prob_{n,0,\omega}\Big[M_n>c\log n\Big|L_n=l\Big]\prob^U_{n,\omega}[L_n=l]\\
  \nonumber
  &\le \sup_{l\in nR\cap\N}\prob_{n,0,\omega}\Big[M_n>c\log n\Big|L_n=l\Big]\,.
\qedhere
\end{align}
\end{proof}

\medskip

\appendix

\section{Proof of Proposition \ref{prop:no_disorder}}
\label{proof:no disorder}

Given integers $n\ge 1$ and $l\in\{1,\ldots,n\}$, the conditions
$L_n=l$ and $S:=\{S_i\}_{i\in\N_0}\ni n$ are tantamount to $S_l=n$, so
we can state that for all $h\in\Rl$
\begin{align}
  \nonumber
  \prob_{n,h,0}\big[M_n\in\cdot\,\big|L_n=l\big]&=\frac{\prob_{n,h,0}[M_n\in\cdot\,,\,L_n=l]}{\prob_{n,h,0}[L_n=l]}\\
  \nonumber
  &=\frac{\prob[M_n\in\cdot\,,\,S_l=n]}{\prob[S_l=n]}=\prob\big[M_n\in\cdot\,\big|S_l=n\big]\,.
\end{align}
We shall refer to the last expression when computing conditional
probabilities for the pure model. Moreover, given two sequences
$\{a_n\}_{n\in\N}$ and $\{b_n\}_{n\in\N}$ of positive real numbers,
below we write $a_n\sim b_n$ if the two sequences are asymptotically
equivalent, i.e., if $\lim_{n\uparrow\infty}a_n/b_n=1$.

\smallskip

\noindent \textit{Proof of part (i).}  Fix a closed set
$R\subset(\rho_c,1)$ and a number $\epsilon>0$. For every $n\in\N$ let
$l_n\in nR\cap\N$ be such that
\begin{equation*}
  \sup_{l\in nR\cap\N}\,\prob\bigg[\bigg|\frac{M_n}{\log n}-\frac{1}{f\circ\imath_\rho(l/n)}\bigg|>\epsilon\bigg|S_l=n\bigg]=
  \prob\bigg[\bigg|\frac{M_n}{\log n}-\frac{1}{f\circ\imath_\rho(l_n/n)}\bigg|>\epsilon\bigg|S_{l_n}=n\bigg]\,.
\end{equation*}
The number $l_n$ exists because $nR\cap\N$ is a finite set. Thus,
putting $\xi_n:=f\circ\imath_\rho(l_n/n)$ for brevity, part $(i)$ of
the proposition is demonstrated if we show that
\begin{equation}
  \label{eq:big-jump-1}
  \lim_{n\uparrow\infty}\prob\big[M_n\le(1/\xi_n+\epsilon)\log n\big|S_{l_n}=n\big]=1
\end{equation}
and
\begin{equation}
  \label{eq:big-jump-2}
   \lim_{n\uparrow\infty}\prob\big[M_n<(1/\xi_n-\epsilon)\log n\big|S_{l_n}=n\big]=0\,.
\end{equation}
We stress that $\rho_c<\inf R\le l_n/n\le \sup R<1$ and that
$0<f\circ\imath_\rho(\inf R)\le\xi_n\le f\circ\imath_\rho(\sup R)<+\infty$ for
all $n\in\N$.

The proof of \eqref{eq:big-jump-1} and \eqref{eq:big-jump-2} relies on
the following asymptotic equivalence, which holds for any real
sequence $\{\lambda_n\}_{n\in\N}$ that increases to infinity fast
enough to assure that
$\lim_{n\uparrow\infty}\sqrt{n}\,\Ex[\mathds{1}_{\{T_1>\lambda_n\}}T_1\ee^{-\xi_nT_1}]=0$:
\begin{equation}
 \prob\big[M_n\le\lambda_n\big|S_{l_n}=n\big]
  \sim\bigg(1-\frac{\Ex[\mathds{1}_{\{T_1>\lambda_n\}} \ee^{-\xi_n T_1}]}{\Ex[\ee^{-\xi_n T_1}]}\bigg)^{\!l_n}\,.
     \label{eq:big-jump-3}
\end{equation}
In order to verify \eqref{eq:big-jump-3}, for $n\in\N$ consider the
probability mass functions $p_n$ and $q_n$ on $\N$ defined by
\begin{equation*}
    p_n(t):=\frac{\ee^{-\xi_n t}p(t)}{\Ex[\ee^{-\xi_nT_1}]}
\end{equation*}
and
  \begin{equation*}
    q_n(t):=\frac{\mathds{1}_{\{t\le \lambda_n\}}\ee^{-\xi_n t}p(t)}{\Ex[\mathds{1}_{\{T_1\le \lambda_n\}} \ee^{-\xi_nT_1}]}
\end{equation*}
  for all $t\in\N$. Since $0<f\circ\imath_\rho(\inf R)\le\xi_n\le
  f\circ\imath_\rho(\sup R)<+\infty$, all moments of $p_n$ and $q_n$ exist
  and they are uniformly bounded with respect to $n$.  The identity
  $\rho(h)=\Ex[\ee^{-f(h)T_1}]/\Ex[T_1\ee^{-f(h)T_1}]$ for $h>h_c$
  allows one to deduce that the mean of $p_n$ is $n/l_n$. We denote by
  $v_n$ the variance of $p_n$ and by $m_n$ and $w_n$ the mean and the
  variance of $q_n$, respectively.  One can easily verify that
  $\liminf_{n\uparrow\infty}v_n>0$ and
  $\liminf_{n\uparrow\infty}w_n>0$, so the local CLT for triangular
  arrays of integer random variables, which can be obtained by
  combining the Lindeberg--Feller CLT (see \cite[Theorem
    27.2]{billingsley1986}) with \cite[Theorem 1.2]{davis1995},
  applies to random variables distributed according to $p_n$ and
  $q_n$, giving
  \begin{equation*}
  \adjustlimits\lim_{n\uparrow\infty}\sup_{t\in\N}\bigg|\sqrt{2\pi l_n v_n}\,p_n^{\star l_n}(t)-\ee^{-\frac{(t-n)^2}{2l_n v_n}}\bigg|=0
  \end{equation*}
  and
  \begin{equation*}
  \adjustlimits\lim_{n\uparrow\infty}\sup_{t\in\N}\bigg|\sqrt{2\pi l_n w_n}\,q_n^{\star l_n}(t)-\ee^{-\frac{(t-l_nm_n)^2}{2l_n w_n}}\bigg|=0\,,
\end{equation*}
where $p_n^{\star l}$ and $q_n^{\star l}$ denote the $l$-fold
convolution of $p_n$ and $q_n$, respectively. The first of these
limits shows that
\begin{align}
  \nonumber
  \prob\big[S_{l_n}=n\big]&=\ee^{\xi_n n}\,\Ex\bigg[\mathds{1}_{\{S_{l_n}=n\}}\prod_{i=1}^{l_n}\ee^{-\xi_n T_i}\bigg]\\
  &=\ee^{\xi_n n}\,\Ex\big[\ee^{-\xi_n T_1}\big]^{l_n} p_n^{\star l_n}(n)\sim\frac{\ee^{\xi_n n}}{\sqrt{2\pi l_n v_n}}\,\Ex\big[\ee^{-\xi_n T_1}\big]^{l_n}\,,
\label{eq:big-jump-4}
\end{align}
while the second yields
\begin{align}
  \nonumber
\prob\big[M_n\le \lambda_n,\,S_{l_n}=n\big]&=\ee^{\xi_n n}\,\Ex\bigg[\mathds{1}_{\{S_{l_n}=n\}}\prod_{i=1}^{l_n}\mathds{1}_{\{T_i\le \lambda_n\}} \ee^{-\xi_n T_i}\bigg]\\
\nonumber
&=\ee^{\xi_n n}\,\Ex\Big[\mathds{1}_{\{T_1\le \lambda_n\}} \ee^{-\xi_n T_1}\Big]^{l_n} q_n^{\star l_n}(n)\\
&\sim\frac{\ee^{\xi_n n-\frac{(n-l_nm_n)^2}{2l_n w_n}}}{\sqrt{2\pi l_n w_n}}\,\Ex\Big[\mathds{1}_{\{T_1\le \lambda_n\}} \ee^{-\xi_n T_1}\Big]^{l_n}\,.
\label{eq:big-jump-5}
\end{align}
Dividing \eqref{eq:big-jump-5} by \eqref{eq:big-jump-4} we obtain
\eqref{eq:big-jump-3} because $\lim_{n\uparrow\infty}w_n/v_n=1$ and,
thanks to the hypothesis that
$\lim_{n\uparrow\infty}\sqrt{n}\,\Ex[\mathds{1}_{\{T_1>\lambda_n\}}T_1\ee^{-\xi_nT_1}]=0$,
because $\lim_{n\uparrow\infty}(n-l_nm_n)/\sqrt{n}=0$.  In fact, for
all $n\in\N$ we have
\begin{equation*}
  n-l_nm_m
  =\frac{l_n\Ex[\mathds{1}_{\{T_1> \lambda_n\}} T_1\ee^{-\xi_nT_1}]-n\Ex[\mathds{1}_{\{T_1>\lambda_n\}} \ee^{-\xi_nT_1}]}{\Ex[\mathds{1}_{\{T_1\le \lambda_n\}} \ee^{-\xi_nT_1}]}\,.
\end{equation*}

We are now ready to prove \eqref{eq:big-jump-1} and
\eqref{eq:big-jump-2}. We get \eqref{eq:big-jump-1} from
\eqref{eq:big-jump-3} once the choice
$\lambda_n=(1/\xi_n+\epsilon)\log n$ is made. Indeed, such $\lambda_n$
entails that both the limits
$\lim_{n\uparrow\infty}\sqrt{n}\,\Ex[\mathds{1}_{\{T_1>\lambda_n\}}T_1
  \ee^{-\xi_n T_1}]=0$ and
$\lim_{n\uparrow\infty}l_n\Ex[\mathds{1}_{\{T_1>\lambda_n\}}
  \ee^{-\xi_n T_1}]=0$ hold true, which are immediate from the bound
$\Ex[\mathds{1}_{\{T_1>\lambda_n\}} T_1\ee^{-\xi_n
    T_1}]\le(\lfloor\lambda_n\rfloor+1)\ee^{-\xi_n\lambda_n}/(1-\ee^{-\xi_n})^2$
valid for all $n\in\N$. Regarding \eqref{eq:big-jump-2}, suppose
without restriction that $\epsilon f\circ\imath_\rho(\sup R)<1/2$ and take
$\lambda_n=(1/\xi_n-\epsilon)\log n$ in \eqref{eq:big-jump-3}. Since
$\xi_n\lambda_n\ge[1-\epsilon f\circ\imath_\rho(\sup R)]\log n$ with
$1-\epsilon f\circ\imath_\rho(\sup R)>1/2$, we find
$\lim_{n\uparrow\infty}\sqrt{n}\,\Ex[\mathds{1}_{\{T_1>\lambda_n\}}T_1
  \ee^{-\xi_n T_1}]=0$, as it should be. Then, \eqref{eq:big-jump-3}
implies \eqref{eq:big-jump-2} since
$\lim_{n\uparrow\infty}l_n\Ex[\mathds{1}_{\{T_1>\lambda_n\}}
  \ee^{-\xi_n T_1}]=+\infty$, which is due to the inequality
$\Ex[\mathds{1}_{\{T_1>\lambda_n\}} \ee^{-\xi_n T_1}]\ge
p(\lfloor\lambda_n\rfloor+1)\ee^{-\xi_n(\lfloor\lambda_n\rfloor+1)}$
valid for all $n\in\N$ and to the fact that
$\lim_{n\uparrow\infty}(\log
n)^{\alpha+2}p(\lfloor\lambda_n\rfloor+1)=+\infty$ (see
\cite[Proposition 1.3.6]{bingham1989}).

\smallskip

\noindent \textit{Proof of part (ii).}  Suppose that $\rho_c>0$ and
fix a closed set $R\subset(0,\rho_c)$ and a number $\epsilon>0$. As
before, let $l_n\in nR\cap\N$ be such that
\begin{equation*}
  \sup_{l\in nR\cap\N}\,\prob\bigg[\bigg|\frac{M_n}{n}-\bigg(1-\frac{l/n}{\rho_c}\bigg)\bigg|>\epsilon\bigg|S_l=n\bigg]=
  \prob\bigg[\bigg|\frac{M_n}{n}-\bigg(1-\frac{l_n/n}{\rho_c}\bigg)\bigg|>\epsilon\bigg|S_{l_n}=n\bigg]\,.
\end{equation*}
Part $(ii)$ of the proposition is demonstrated if we show that
\begin{equation}
  \label{eq:big-jump-100}
 \lim_{n\uparrow\infty}\prob\big[M_n>n-l_n/\rho_c+\epsilon n,\big|S_{l_n}=n\big]=0
\end{equation}
and 
\begin{equation}
  \label{eq:big-jump-101}
 \lim_{n\uparrow\infty}\prob\big[M_n\ge n-l_n/\rho_c-\epsilon n,\big|S_{l_n}=n\big]=1\,.
\end{equation}

To begin with, we note that $\rho_c>0$ means
$\Ex[T_1]=1/\rho_c<+\infty$, so we have the following local LDP (see
\cite[Theorem 1]{doney1989}): for all $\delta>0$
\begin{equation}
  \label{precise_LDP}
  \adjustlimits\lim_{l\uparrow\infty}\sup_{n\ge (1/\rho_c+\delta)l}\bigg|\frac{\prob[S_l=n]}{lp(\lfloor n-l/\rho_c\rfloor)}-1\bigg|=0\,.
\end{equation}
Since the condition $l_n\in nR\cap\N$ implies that
$(1/\rho_c+\delta)l_n\le n$ with $\delta:=1/\sup R-1/\rho_c>0$,
(\ref{precise_LDP}) shows that $\prob[S_{l_n}=n]\sim l_n p(\lfloor
n-l_n/\rho_c\rfloor)$.  Exploiting the fact that $\ell$ is assumed to
be a slowly varying function at infinity defined on $[1,+\infty)$, in
  order to simplify the formulas we extend $p$ to all real numbers
  $z\ge 1$ by setting $p(z):=\ell(z)/z^{\alpha+1}$. In this way,
  \cite[Theorem 1.5.2]{bingham1989} allows us to conclude that
  \begin{equation*}
    \prob\big[S_{l_n}=n\big]\sim l_n p\big(n-l_n/\rho_c\big)\,.
  \end{equation*}
Setting $z_n:=n-l_n/\rho_c$ for brevity, \eqref{eq:big-jump-100} and
\eqref{eq:big-jump-101} become then equivalent to the following
results to be verified:
\begin{equation}
  \label{eq:big-jump-102}
 \lim_{n\uparrow\infty}\frac{\prob[M_n>z_n+\epsilon n,\,S_{l_n}=n]}{l_np(z_n)}=0
\end{equation}
and 
\begin{equation}
  \label{eq:big-jump-103}
  \liminf_{n\uparrow\infty}\frac{\prob[M_n>z_n-\epsilon n,\,S_{l_n}=n]}{l_np(z_n)}\ge 1\,.
\end{equation}

Let us prove \eqref{eq:big-jump-102} and \eqref{eq:big-jump-103}. Here
and below we tacitly use the fact that $l_n\le n$, as well as the
facts that $\lim_{n\uparrow\infty}l_n=\infty$ and
$\lim_{n\uparrow\infty}z_n=+\infty$.  For every $n\in\N$ we have
\begin{align}
  \nonumber
 \prob\big[M_n>z_n+\epsilon n,\,S_{l_n}=n\big]&\le\prob\Big[\max\{T_1,\ldots,T_{l_n}\}>z_n+\epsilon l_n,\,S_{l_n}=n\Big]\\
  \nonumber
  &\le \sum_{i=1}^{l_n}\prob\big[T_i> z_n+\epsilon l_n,\,S_{l_n}=n\big]\\
   \nonumber
  &=l_n\sum_{t\in\N}\mathds{1}_{\{t> z_n+\epsilon l_n\}}p(t)\prob\big[S_{l_n-1}=n-t\big]\\
   &\le l_n\prob\big[S_{l_n-1}\le n-z_n-\epsilon l_n\big]\sup_{z\ge z_n}p(z)\,.
   \label{eq:big-jump-104}
\end{align}
But $\sup_{z\ge z_n}p(z)\sim p(z_n)$ (see \cite[Theorem
  1.5.3]{bingham1989}) and the law of large numbers assures us that
$\lim_{n\uparrow\infty}\prob[S_{l_n-1}\le l_n/\rho_c-\epsilon l_n
]=0$. Thus, \eqref{eq:big-jump-104} immediately yields
\eqref{eq:big-jump-102}.

Regarding \eqref{eq:big-jump-103}, we suppose without restriction that
$2\epsilon\le 1-\sup R/\rho_c$. This gives $z_n:=n-l_n/\rho_c\ge
n(1-\sup R/\rho_c)\ge 2\epsilon n$ and, as a consequence,
$z_n-\epsilon n=z_n/2+z_n/2-\epsilon n\ge z_n/2$, so for every
$n\in\N$ we can state that
\begin{align}
  \nonumber
   \prob\big[M_n\ge z_n-\epsilon n,\,S_{l_n}=n\big]&\ge \sum_{i=1}^{l_n}\prob\big[T_i\ge z_n-\epsilon n,\,S_{l_n}=n\big]\\
  \nonumber
  &\quad-\sum_{i=1}^{l_n-1}\sum_{j=i+1}^{l_n}\prob\Big[T_i\ge z_n-n\epsilon,\,T_j\ge z_n-\epsilon n,\,S_{l_n}=n\Big]\\
  \nonumber
  &\ge l_n\prob\big[T_{l_n}\ge z_n-\epsilon l_n,\,S_{l_n}=n\big]\\
  &\quad -l_n^2\prob\Big[T_{l_n-1}\ge z_n/2,\,T_{l_n}\ge z_n/2,\,S_{l_n}=n\Big]\,.
  \label{eq:big-jump-105}
\end{align}
The last term in \eqref{eq:big-jump-105} can be bound as follows:
\begin{align}
  \nonumber
  &\prob\Big[T_{l_n-1}\ge z_n/2,\,T_{l_n}\ge z_n/2,\,S_{l_n}=n\Big]\\
  \nonumber
  &\qquad=\sum_{t'\in\N}\sum_{t''\in\N}\mathds{1}_{\{t'\ge z_n/2\}}\mathds{1}_{\{t''\ge z_n/2\}}p(t')p(t'')\prob\big[S_{l_n-2}=n-t'-t''\big]\\
  \nonumber
  &\qquad\le \Big[\sup_{z\ge z_n/2}p(z)\Big]^2\sum_{t\in\N}(t-1)\prob\big[S_{l_n-2}=n-t\big]\le n\Big[\sup_{z\ge z_n/2}p(z)\Big]^2\,.
\end{align}
So, using that $\sup_{z\ge z_n/2}p(z)\sim p(z_n/2)\sim
2^{\alpha+1}p(z_n)=(2/z_n)^{\alpha+1}\ell(z_n)$ (see \cite[Theorems
  1.5.3 and 1.5.2]{bingham1989}), we see that
\begin{equation*}
  \lim_{n\uparrow\infty} \frac{l_n\prob[T_{l_n-1}\ge z_n/2,\,T_{l_n}\ge z_n/2,\,S_{l_n}=n]}{p(z_n)}=0
\end{equation*}
because either $\alpha>1$ and $\ell(z)\le z^{(\alpha-1)/2}$ for all
sufficiently large $z$ (see \cite[Proposition 1.3.6]{bingham1989}) or
$\alpha=1$ and $\lim_{z\uparrow+\infty}\ell(z)=0$ as
$\Ex[T_1]<+\infty$. In this way, in order to prove
\eqref{eq:big-jump-103} from \eqref{eq:big-jump-105} it remains to
show that
\begin{equation}
  \liminf_{n\uparrow\infty} \frac{\prob[T_{l_n}\ge z_n-\epsilon l_n,\,S_{l_n}=n]}{p(z_n)}\ge 1\,.
\label{eq:big-jump-106}
\end{equation}
To this aim, pick a number $\delta>1$ and, recalling that $z_n\ge
2\epsilon n$, note that $z_n+2\epsilon(\delta-1) l_n\le \delta
z_n$. For all $n\in\N$ we have
\begin{align}
  \nonumber
  &\prob\big[T_{l_n}\ge z_n-\epsilon l_n,\,S_{l_n}=n\big]\\
  \nonumber
  &\qquad=\sum_{t\in\N}\mathds{1}_{\{t\ge z_n-\epsilon l_n\}}p(t)\prob\big[S_{l_n-1}=n-t\big]\\
 \nonumber
 &\qquad\ge \sum_{t\in\N}\mathds{1}_{\{z_n-\epsilon l_n\le t\le z_n+2\epsilon(\delta-1) l_n\}}p(t)\prob\big[S_{l_n-1}=n-t\big]\\
 \nonumber
  &\qquad\ge  \prob\Big[n-z_n-2\epsilon(\delta-1) l_n\le S_{l_n-1}\le n-z_n+\epsilon l_n\Big]\inf_{z\in[1,\delta z_n]}p(z)\,.
\end{align}
The law of large numbers yields
$\lim_{n\uparrow\infty}\prob[l_n/\rho_c-2\epsilon(\delta-1) l_n\le
  S_{l_n-1}\le l_n/\rho_c+\epsilon l_n]=1$, so using that
$\inf_{z\in[1,\delta z_n]}p(z)\sim p(\delta z_n)\sim
p(z_n)/\delta^{\alpha+1}$ (see \cite[Theorems 1.5.3 and
  1.5.2]{bingham1989}) we find
\begin{equation*}
  \liminf_{n\uparrow\infty} \frac{\prob[T_{l_n}\ge z_n-\epsilon l_n,\,S_{l_n}=n]}{p(z_n)}\ge\frac{1}{\delta^{\alpha+1}}\,.
\end{equation*}
This demonstrates \eqref{eq:big-jump-106} because $\delta$ can be
taken arbitrarily close to 1.

\section*{Acknowledegements}
G.G. acknowledges the support of the Cariparo Foundation.

\end{document}